\documentclass[11pt,reqno,twoside]{amsart}
\usepackage{mathrsfs}
\usepackage{amsfonts,amssymb,amsmath,amsthm}

\usepackage{cite}
\usepackage[colorlinks=true,citecolor=red,linkcolor=blue]{hyperref}
\newcommand{\secref}[1]{Section~\ref{#1}}
\newcommand{\subsecref}[1]{Subsection~\ref{#1}}
\newcommand{\thmref}[1]{Theorem~\ref{#1}}
\newcommand{\lemref}[1]{Lemma~\ref{#1}}
\newcommand{\propref}[1]{Proposition~\ref{#1}}
\newcommand{\corref}[1]{Corollary~\ref{#1}}
\usepackage{titletoc}
\usepackage{geometry}
\usepackage{bm}
\usepackage{indentfirst}
\usepackage{graphicx}
\usepackage{stmaryrd}
\usepackage{tikz}

\usepackage{color,soul}
\usepackage{setspace}
\usepackage[pagewise]{lineno}
\usepackage{float}

\usepackage{todonotes}

\def\T {{\mathbb T}}
\def\C {{\mathbb C}}

\def\R {\mathbb{R}}

\def\Z {\mathbb{Z}}
\def\d{{\,\rm d}}
\numberwithin{equation}{section}
\newtheorem{theorem}{Theorem}[section]
\newtheorem{lemma}[theorem]{Lemma}
\newtheorem{corollary}[theorem]{Corollary}
\newtheorem{proposition}[theorem]{Proposition}

\newtheorem*{definition*}{Definition}

\newtheorem{remark}[theorem]{Remark}
\theoremstyle{definition}

\begin{document}

\title[Controllable subspaces]{Controllable subspaces and real-part observability for Schr\"odinger equations on \(\T^d\)}

\author{Gengsheng Wang}

\address{HETAO Institute of Mathematics and Interdisciplinary Studies (Shenzhen), 518017, China}
\email{wanggs62@yeah.net}

 \author{Ming Wang}
\address{School of Mathematics and Statistics, HNP-LAMA, Central South University, Changsha, Hunan 410083, China}
\email{m.wang@csu.edu.cn}


\date{\today}
\begin{abstract}
We study internal controllability of Schr\"odinger equations on tori with
controls acting only through their real parts. This real-part constraint leads
to a real-linear control problem for which full null controllability may fail.
We characterize the maximal null-controllable subspace and show that the
uncontrollable directions are precisely given by the purely imaginary
stationary modes. These results provide an explicit description of
controllable and uncontrollable directions for Schr\"odinger equations with
real-part controls.

The characterization is obtained through observability inequalities with sharp
stationary correction terms. For cylindrical open control regions, we prove
such an estimate by a real-part compactness--uniqueness argument. For
admissible measurable product control regions in the shifted free case, we
establish the corresponding estimate for rough spacetime control sets. The
proof combines complex-valued observability inputs with a double-spectrum
analysis arising from the coupling of the forward and backward Schr\"odinger
evolutions in the real-part observation.
\end{abstract}

\maketitle
\tableofcontents
\section{Introduction and main results}
\label{sec:introduction}
\subsection{Motivation and problem}
\label{subsec:motivation}

Controllability of linear systems is a fundamental topic in control theory.
For partial differential equations, much of the classical literature focuses on
the question of whether every initial state can be driven to zero by a
suitable control. In this paper, we consider a finer structural problem:
when full null controllability fails, which initial states remain
controllable? Equivalently, how can one characterize the controllable
subspace of the system?

For a well-posed linear control system with state space \(H\) and free evolution
\(S(t)\), the controllable subspace at time \(T>0\) consists of all initial data
that can be driven to zero at time \(T\). If \(\mathcal R_T\) denotes the
reachable space from the origin at time \(T\), then an initial datum \(y_0\)
belongs to the controllable subspace at time \(T\) if and only if
\[
	S(T)y_0\in\mathcal R_T .
\]
Thus, the characterization of controllable subspaces is closely related to the
description of reachable spaces. In finite-dimensional systems, Kalman's theory
gives an explicit characterization of reachable and controllable subspaces. In
contrast, such a characterization is much more subtle for infinite-dimensional
systems governed by partial differential equations; see, for instance,
\cite{Tucsnak2022,DardeErvedoza2018,HartmannKellayTucsnak2020}.

The characterization of controllable subspaces in infinite-dimensional systems
is also closely related to finite-codimensionality in optimal control theory;
see \cite{LiYong1995}. Rather than requiring controllability in the whole state
space, one asks whether the uncontrollable directions can be characterized as
a finite-dimensional obstruction. Such conditions are closely connected with
G{\aa}rding-type estimates for the adjoint system; see
\cite{LiuLuZhang2020}.

We study the characterization of controllable subspaces for Schr\"odinger
equations on the torus with real-part controls. Here the control enters the
equation as a forcing term only through its real part. The system is therefore
linear in the state variable, but the control map is real-linear rather than
complex-linear. The natural framework is consequently the real Hilbert
structure of \(L^2(\mathbb T^d;\mathbb C)\). In this setting, the failure of
full null controllability may result from a structural obstruction.

The main purpose of this paper is to identify the maximal null-controllable
subspaces for Schr\"odinger systems with real-part controls in different
geometric settings. Our main results provide explicit characterizations of
these subspaces and identify the stationary purely imaginary modes as the
precise obstruction preventing full controllability. Thus, the results give a
structural description of controllable and uncontrollable directions, rather
than only providing criteria for full controllability.

\subsection{Main results}
\label{subsec:main-results}

We now give the precise formulation. Denote by
\[
\T^d=\R^d/(2\pi\Z)^d
\]
the \(d\)-dimensional torus. Fix \(T>0\), a measurable set
\(G\subset (0,T)\times\T^d\), and a real-valued continuous potential
\(V\in C(\T^d;\R)\). We consider the controlled Schr\"odinger equation
\begin{equation}\label{equ-sch-control}
	\left\{
	\begin{aligned}
		\partial_tu
		&=
		i(\Delta-V(x))u+\mathbf 1_G(t,x)\operatorname{Re} f(t,x),
		\quad (t,x)\in (0,T)\times \T^d,\\
		u|_{t=0}&=u_0(x).
	\end{aligned}
	\right.
\end{equation}
Here \(\mathbf 1_G\) denotes the characteristic function of \(G\), and
\(f\in L^2((0,T)\times\T^d;\C)\) is the control. Only
\(\operatorname{Re} f\) acts on the system. Thus, the effective control space
is the real-valued space
\[
L^2((0,T)\times\T^d;\R).
\]
Throughout the controllability and duality arguments, the state space
\(L^2(\T^d;\C)\) is regarded as a real Hilbert space, endowed with the real
part of the usual complex \(L^2\)-inner product:
\begin{equation}\label{equ-Re-inner}
	\langle u,v\rangle_{\R}
	=
	\operatorname{Re}\int_{\T^d}u(x)\overline{v(x)}\,\d x,
	\qquad u,v\in L^2(\T^d;\C).
\end{equation}

For \eqref{equ-sch-control}, we denote by \(X\) the controllable subspace at
time \(T\). Namely, \(X\) is the set of all initial data
\(u_0\in L^2(\T^d;\C)\) for which there exists a control
\(f\in L^2((0,T)\times\T^d;\C)\) such that the corresponding solution of
\eqref{equ-sch-control} satisfies \(u(T)=0\). Since only
\(\operatorname{Re}f\) acts on the system, \(X\) is a real-linear subspace of
\(L^2(\T^d;\C)\), but it need not be a complex-linear subspace.

  To formulate the main results, we first describe the stationary obstruction.
Let \(\mathcal K\) be the stationary kernel of the Schr\"odinger operator,
namely
\begin{equation}\label{1.3-7-14-w}
\mathcal K=\ker(-\Delta+V),
\end{equation}
which is invariant under complex conjugation since \(V\) is real-valued. We
define the purely imaginary stationary subspace by
\begin{equation}\label{equ-ker-ReIm}
	\mathcal K_{\operatorname{Im}}
	=
	\{i\,\operatorname{Im}\phi:\phi\in\mathcal K\}.
\end{equation}
The real orthogonal complement of \(\mathcal K_{\operatorname{Im}}\) with
respect to the inner product \eqref{equ-Re-inner} is denoted by
\begin{equation}\label{1.5-7-14-w}
\mathcal K_{\operatorname{Im}}^{\perp_\R}
=
\{u\in L^2(\T^d;\C):
\langle u,v\rangle_\R=0
\text{ for all }v\in\mathcal K_{\operatorname{Im}}\}.
\end{equation}

Our first main result identifies the controllable subspace for cylindrical open
control regions.

\begin{theorem}[Controllable subspace for cylindrical open regions]
	\label{thm-controlspace}
	Let \(X\) be the controllable subspace at time \(T\) for
	\eqref{equ-sch-control}, with \(G=(0,T)\times E\), where
	\(E\subset\T^d\) is nonempty and open. Then
	\[
	X
	=
	\mathcal K_{\operatorname{Im}}^{\perp_\R},
	\]
	where \(\mathcal K_{\operatorname{Im}}^{\perp_\R}\) is given by
	\eqref{1.5-7-14-w}.
\end{theorem}

The significance of \thmref{thm-controlspace} is that it gives an explicit
characterization of the maximal null-controllable subspace. In particular, it
identifies the precise obstruction to controllability for the Schr\"odinger
system with real-part controls.

Several comments on \thmref{thm-controlspace} are in order.

\begin{itemize}
	\item[(i)]
	The formula shows that \(X\) is independent of \(T\) and of the particular
	nonempty open set \(E\), but depends on \(V\) through the stationary kernel
	\(\mathcal K=\ker(-\Delta+V)\). Thus, for cylindrical open control regions,
	the control geometry introduces no additional obstruction. Equivalently, the
	reachable set from the origin at time \(T\) is the same real subspace
	\(\mathcal K_{\operatorname{Im}}^{\perp_\mathbb R}\).

	\item[(ii)]
	If \(V\geq 0\) and
	\[
	\int_{\T^d}V(x)\,\d x>0,
	\]
	then the standard energy identity implies
	\(\mathcal K=\{0\}\); see Lemma \ref{lem-positive-V}. Hence, by
	\thmref{thm-controlspace},
	\[
	X=L^2(\T^d;\C).
	\]

	\item[(iii)]
	If \(V=0\), then \(\mathcal K\) consists of the constant functions. Hence
	\thmref{thm-controlspace} gives
	\[
	X
	=
	\left\{
	\phi\in L^2(\T^d;\C):
	\operatorname{Im}\int_{\T^d}\phi(x)\,\d x=0
	\right\}.
	\]
	Thus the only obstruction is the imaginary part of the spatial average.
\end{itemize}

Our second result treats the special case \(V\equiv -c\), namely the
shifted free Schr\"odinger equation, with an admissible measurable
product control region in spacetime. We set
\begin{equation}\label{eq:time-slab-Q}
	Q=(0,2\pi)\times\mathbb T^d.
\end{equation}
For \(c\in\R\), we consider the shifted free controlled equation
\begin{equation}\label{equ-shifted-control}
  \left\{
  \begin{aligned}
    \partial_tu
    &=
    i(\Delta+c)u+\mathbf 1_G(t,x)\operatorname{Re} f(t,x),
    \quad (t,x)\in Q,\\
    u|_{t=0}&=u_0(x).
  \end{aligned}
  \right.
\end{equation}

The admissible control regions are defined as follows.
We write
\begin{equation}\label{1.8-7-27-w}
z=(z_0,\ldots,z_d)=(t,x_1,\ldots,x_d),
\end{equation}
so that the time variable has index \(0\). A measurable set
\(G\subset Q\) is called an admissible measurable product set if the index set \(\{0,\ldots,d\}\) can be partitioned into blocks
\begin{equation}\label{1.9-7-27-w}
\{0,\ldots,d\}=\alpha_1\sqcup\cdots\sqcup\alpha_m,
\qquad
1\leq \#\alpha_j\leq 2 .
\end{equation}
For each block \(\alpha_j\), we denote the corresponding variables by
\begin{equation}\label{eq:block-variable-definition}
z_{\alpha_j}:=(z_\ell)_{\ell\in\alpha_j}.
\end{equation}
We assume that there exist measurable sets
\(G_j\subset Q_{\alpha_j}\), each of positive measure in its coordinate
space, such that
\begin{equation}\label{1.10-7-27-w}
G=\prod_{j=1}^m G_j,
\end{equation}
where
\[
Q_0=(0,2\pi),\qquad Q_\ell=\mathbb T\quad(\ell\geq1),
\qquad
Q_{\alpha_j}:=\prod_{\ell\in\alpha_j}Q_\ell .
\]

To describe the obstruction in this shifted free case, let
\begin{equation}\label{1.7-7-14-w}
S=\{n\in\Z^d:\ |n|^2-c=0\},
\end{equation}
and let \(\mathcal K_c\) be the span of the corresponding Fourier modes:
\begin{equation}\label{1.8-7-14-w}
\mathcal K_c=\operatorname{span}\{e^{i n\cdot x}:n\in S\}.
\end{equation}
We define
\begin{equation}\label{1.9-7-14-w}
(\mathcal K_c)_{\operatorname{Im}}
=
\{i\,\operatorname{Im}\phi:\phi\in\mathcal K_c\}.
\end{equation}
The real orthogonal complement of \((\mathcal K_c)_{\operatorname{Im}}\) with
respect to the inner product \eqref{equ-Re-inner} is denoted by
\begin{equation}\label{1.10-7-14-w}
\bigl((\mathcal K_c)_{\operatorname{Im}}\bigr)^{\perp_\R}
=
\left\{
u\in L^2(\T^d;\C):
\langle u,v\rangle_\R=0
\text{ for all }v\in(\mathcal K_c)_{\operatorname{Im}}
\right\}.
\end{equation}
We note that \eqref{equ-shifted-control} is the special case
\(V\equiv -c\) of \eqref{equ-sch-control}. Under this specialization,
\[
\mathcal K
=
\ker(-\Delta-c)
=
\mathcal K_c.
\]
Consequently, the spaces \((\mathcal K_c)_{\operatorname{Im}}\) and \(\bigl((\mathcal K_c)_{\operatorname{Im}}\bigr)^{\perp_\R}\) defined above coincide with the spaces
\(\mathcal K_{\operatorname{Im}}\) and
\(\mathcal K_{\operatorname{Im}}^{\perp_\R}\), respectively, defined in
\eqref{equ-ker-ReIm} and \eqref{1.5-7-14-w}, in the special case
\(V\equiv -c\).

\begin{theorem}[Controllable subspace for admissible measurable product regions]
	\label{thm:measurable-control}
	Let \(X\) be the controllable subspace at time \(2\pi\) for
	\eqref{equ-shifted-control}, where \(G\subset Q\) is an admissible measurable
	product set. Then
	\[
	X
	=
	\bigl((\mathcal K_c)_{\operatorname{Im}}\bigr)^{\perp_\R},
	\]
	where \(\bigl((\mathcal K_c)_{\operatorname{Im}}\bigr)^{\perp_\R}\) is
	defined by \eqref{1.10-7-14-w}.
\end{theorem}

The shifted free setting allows one to exploit the explicit lattice dispersion
relation
\[
\mu_n=|n|^2-c,
\]
which is essential in separating the two spectral branches arising from the
real-part observation on measurable spacetime product sets.

Several comments on \thmref{thm:measurable-control} are in order.

\begin{itemize}
	\item[(i)]
	The formula shows that \(X\) is independent of the particular admissible
	measurable product set \(G\). Thus, within this class of control regions,
	the geometry of \(G\) creates no additional obstruction. Equivalently, the
	reachable set from the origin at time \(2\pi\) is the same real subspace
	\(\bigl((\mathcal K_c)_{\operatorname{Im}}\bigr)^{\perp_\R}\).

		\item[(ii)]
	The dependence of the obstruction on \(c\) is entirely determined by the
	stationary set
	\[
	S=\{n\in\Z^d:\ |n|^2=c\}.
	\]
	Indeed, by \thmref{thm:measurable-control}, the obstruction is precisely
	the purely imaginary part of the stationary eigenspace generated by these
	Fourier modes.

	\item[(iii)]
	In particular,
	\[
	X=L^2(\T^d;\C)
	\]
	if and only if
	\[
	c\notin\{|n|^2:n\in\Z^d\}.
	\]
	This is equivalent to \(S=\emptyset\), and hence to
	\(\mathcal K_c=\{0\}\). Thus the shifted free equation is controllable in
	the whole space precisely when there are no stationary Fourier modes.
\end{itemize}

\subsection{Novelty and relation to previous work}

The main novelty lies in the explicit identification of the maximal
null-controllable subspace and in the proof of sharp real-part
observability inequalities with stationary correction terms.

The controllability problem is naturally connected with observability through
duality. Since the control acts only through its real part, the state space is
considered as a real Hilbert space, and the controllable-subspace
characterization is reduced to observability estimates for the real part of
the adjoint Schr\"odinger evolution. This real-part observation differs from
the usual complex-valued observation, since
\[
\operatorname{Re}(e^{it(\Delta+c)}\varphi)
=
\frac12\left(
e^{it(\Delta+c)}\varphi+
e^{-it(\Delta+c)}\overline{\varphi}
\right),
\]
which couples the forward and backward Schr\"odinger evolutions.

\thmref{thm-controlspace} gives the controllable subspace for cylindrical open
control regions. Its proof develops a real-part compactness--uniqueness
framework, which differs from the standard complex-valued argument. The key
point is to characterize the kernel of the real-part observation and to show
that no additional obstruction arises from the coupling of the two spectral
components. This is achieved through a finite-dimensional unique-continuation
mechanism combined with a high-frequency oscillatory estimate.

\thmref{thm:measurable-control} establishes the corresponding
controllable-subspace characterization for admissible measurable product
control regions in the shifted free case. Compared with
\thmref{thm-controlspace}, the operator is simpler, but the measurable
spacetime control region has considerably less regularity. The overall
controllability-observability framework remains the same; however, the
observability estimate requires a different analysis, since the
compactness--uniqueness argument based on unique continuation is no longer
directly applicable. The proof instead exploits the measurable product
structure and the double-spectrum coupling induced by the real-part
observation.

Schr\"odinger equations with real-valued controls have also been studied in
quantum control; see, for instance,
\cite{ChambrionMasonSigalottiBoscain2009,BeauchardLaurent2010,
BoscainCaponigroChambrionSigalotti2012}.
In those works, the control enters the Hamiltonian in an affine way, and the
main focus is on controllability properties of the resulting Schr\"odinger
systems. This differs from the present problem, which concerns the
characterization of controllable subspaces for Schr\"odinger equations with
real-part controls.

Complex-valued observability for Schr\"odinger equations on tori has been
studied by many methods; see, for instance,
\cite{Ingham1936,Kahane1962,Haraux,Jaffard,Komornic,Taufer2023,JamingWW26,
Wunsch2017,Zygmund1972,AlphonseTzvetkov2025,RauchTaylor1975,Macia2010,
AnaJEMS2014,BZworski12,BBZ13JEMS,Bourgain14,BZworski19,LaurentInternel,
TaoZhongkai,LeBalchMartin2023,NiuWangXiang2025}
and the references therein. In particular, Burq--Zhu
\cite{BurqZhu-spacetime-meas,BurqZhuResolvent2025} established complex-valued
observability from admissible measurable product sets and developed the
auxiliary estimates used as inputs in our argument.

The method of Burq--Zhu does not directly apply to the real-part observation
considered here. Their observation concerns the complex-valued evolution
\(e^{it(\Delta+c)}\varphi\), whereas the real-part observation couples the
forward and backward Schr\"odinger evolutions. To overcome this difficulty,
we combine the complex-valued observability results and the auxiliary
estimates developed by Burq--Zhu with a double-spectrum analysis adapted to
this coupling. The interaction between the two spectral branches at high
frequencies is controlled through an approximation argument, while the
remaining finitely many low frequencies are handled by product stability and
translation differences.

\subsection{Strategy of proof}
\label{subsec:strategy}

We briefly describe the main steps of the proof. Since the control acts only
through its real part, the natural framework is the real Hilbert space
structure of \(L^2(\mathbb T^d;\mathbb C)\). Real Hilbert space duality reduces
the controllability problem to observability estimates for the real part of
the adjoint Schr\"odinger evolution. The stationary purely imaginary modes
appear as the unavoidable correction terms in these estimates.

Let \(\mathcal K=\ker(-\Delta+V)\), and let \(\mathcal K_{\operatorname{Im}}\)
be defined by \eqref{equ-ker-ReIm}. If \(\Pi_{\mathcal K}\) denotes the usual
complex \(L^2\)-orthogonal projection onto \(\mathcal K\), we define
\begin{equation}\label{eq:P-Kim-definition}
P_{\mathcal K_{\operatorname{Im}}}u
:=
i\,\operatorname{Im}(\Pi_{\mathcal K}u),
\qquad
u\in L^2(\mathbb T^d;\mathbb C).
\end{equation}
Since \(V\) is real-valued, this is precisely the real orthogonal projection
onto \(\mathcal K_{\operatorname{Im}}\). Similarly, in the shifted free case,
if \(\Pi_S\) denotes the Fourier projection onto the stationary space
\(\mathcal K_c\), we define
\begin{equation}\label{eq:P-KcIm-definition}
P_{(\mathcal K_c)_{\operatorname{Im}}}\phi
:=
i\,\operatorname{Im}(\Pi_S\phi).
\end{equation}

With these notations, the duality argument reduces
\thmref{thm-controlspace} to the observability estimate
\[
\|u_0\|_{L^2}^2
\leq
C\int_0^T\int_E
\left|\operatorname{Re} e^{it(\Delta-V)}u_0\right|^2\,\d x\d t
+
\|P_{\mathcal K_{\operatorname{Im}}}u_0\|_{L^2}^2,
\]
and reduces \thmref{thm:measurable-control} to
\[
\|\phi\|_{L^2_x}^2
\leq
C\left\|\operatorname{Re}e^{it(\Delta+c)}\phi\right\|_{L^2(G)}^2
+
\|P_{(\mathcal K_c)_{\operatorname{Im}}}\phi\|_{L^2_x}^2 .
\]
Once these observability inequalities are established, the corresponding
controllable subspaces follow from real Hilbert space duality.

For cylindrical open control regions, we use a compactness--uniqueness
strategy, in the spirit of \cite{BardosLebeauRauch1992}, adapted to
real-part observations. A high--low frequency reduction is
first performed. The high-frequency part is controlled by combining complex
observability on tori
\cite{AnaJEMS2014,BZworski12,BBZ13JEMS}
with an oscillatory estimate that handles the coupling of the two spectral
components in the real part. The remaining finite-dimensional obstruction is
identified by the unique-continuation argument:
\[
\mathcal N_T
:=
\left\{
u_0\in L^2(\mathbb T^d;\mathbb C):
\operatorname{Re} e^{it(\Delta-V)}u_0=0
\text{ on }(0,T)\times E
\right\}
\]
is shown to satisfy
\[
\mathcal N_T=\mathcal K_{\operatorname{Im}}.
\]

For admissible measurable product control regions, the same duality reduction
leads to a different observability analysis. Writing
\(\mu_n=|n|^2-c\), the real part produces the double spectrum
\begin{equation}\label{1.18-7-28-w}
\Sigma_c
=
\{(-\mu_n,n),(\mu_n,n):n\in\mathbb Z^d\}
\subset\mathbb R\times\mathbb Z^d.
\end{equation}
The complex-valued observability results and auxiliary estimates of
Burq--Zhu are applied to the individual spectral branches. The main additional
step is to control the interaction between the two branches: high frequencies
are handled by an approximation argument, while the finitely many low
frequencies are absorbed through product stability and translation
differences. This yields the required double-spectrum observability estimate.

\section{Duality and stationary obstructions}
\label{sec:duality}

In this section we make the duality mechanism precise. Since the control acts
only through its real part, the problem is naturally posed in the real Hilbert
space \(L^2(\T^d;\C)\) with inner product \eqref{equ-Re-inner}. We first show
that a suitable real-part estimate for the adjoint flow implies that the
corresponding real subspace is controllable, following a standard duality
argument \cite{Lions1988,Coronbook}. We then show that purely imaginary
stationary modes give unavoidable obstructions. Combining these two facts
reduces the proofs of \thmref{thm-controlspace} and
\thmref{thm:measurable-control} to two real-part estimates, which are stated in
\subsecref{subsec:open-reduction} and \subsecref{subsec:measurable-reduction},
respectively, and proved in the following sections.

\subsection{Duality reduction to real-part observability}
\label{subsec:duality-reduction}

Fix \(T>0\). Let \(H\) be a self-adjoint Schr\"odinger operator on \(L^2(\T^d;\C)\) with
real-valued coefficients, and set
\[
	U(t):=e^{-itH},\qquad t\in\R .
\]
Let \(G\subset(0,T)\times\T^d\) be measurable. Recall that \(\mathbf 1_G\)
is its characteristic function. Denote by
\[
	\mathbf I_G:
	L^2((0,T)\times\T^d;\mathbb F)
	\longrightarrow
	L^2((0,T)\times\T^d;\mathbb F),
	\qquad \mathbb F\in\{\R,\C\},
\]
the bounded multiplication operator given by
\begin{equation}\label{eq:def-IG}
	(\mathbf I_G h)(t,x)
	=
	\mathbf 1_G(t,x)h(t,x),
	\qquad
	\text{for a.e. }(t,x)\in(0,T)\times\T^d,
\end{equation}
for every \(h\in L^2((0,T)\times\T^d;\mathbb F)\). For a.e.
\(s\in(0,T)\), we write \((\mathbf I_G h)(s)\) for the
\(L^2(\T^d;\mathbb F)\)-slice represented by
\[
	x\mapsto \mathbf 1_G(s,x)h(s,x),\;\;x\in\T^d.
\]

Since only the real part of the control enters the equation, it is
equivalent to work directly with a real-valued control
\[
g\in L^2((0,T)\times\T^d;\R).
\]
We therefore consider
\begin{equation}\label{equ-general-real-control}
	\left\{
	\begin{aligned}
		\partial_tu
		&=
		-iHu+\mathbf I_Gg,
		&& (t,x)\in(0,T)\times\T^d,\\
		u|_{t=0}&=u_0.
	\end{aligned}
	\right.
\end{equation}
 The following proposition is the
duality principle used in the proof of the main results.
\begin{proposition}[Real-part observability implies controllability of a subspace]
	\label{prop:HUM-sufficient}
	Let \(Y\) be a closed real subspace of \(L^2(\T^d;\C)\). Assume that \(Y\)
	satisfies the following two conditions.
	\begin{enumerate}
		\item \(Y\) is invariant under the Schr\"odinger group \(U(t)\), namely
		\[
			U(t)Y\subset Y,\qquad t\in\R .
		\]

		\item The real-linear operator
		\[
			\Phi_T:
			L^2((0,T)\times\T^d;\R)
			\longrightarrow
			L^2(\T^d;\C)
		\]
		given by
		\begin{equation}\label{2.2-wgs-7-6}
			\Phi_Tg
			:=
			\int_0^T U(T-s)\bigl((\mathbf I_G g)(s)\bigr)\,\d s
			\quad\text{in }L^2(\T^d;\C),
			\qquad
			g\in L^2((0,T)\times\T^d;\R),
		\end{equation}
		satisfies
		\[
			\operatorname{Ran}\Phi_T\subset Y .
		\]
	\end{enumerate}
	Assume moreover that there exists \(C>0\) such that
	\begin{equation}\label{equ-HUM-observability}
		\|\varphi\|_{L^2(\T^d)}^2
		\leq
		C\left\|
		\mathbf I_G\bigl(\operatorname{Re}(U(T-t)^*\varphi)\bigr)
		\right\|_{L^2((0,T)\times\T^d)}^2,\;\;\forall \varphi\in Y
	\end{equation}
	Then for every \(u_0\in Y\), there exists a control
	\(g\in L^2((0,T)\times\T^d;\R)\) such that the corresponding solution \(u\)
	of \eqref{equ-general-real-control} satisfies \(u(T)=0\).
\end{proposition}

\begin{proof}
	Throughout the proof, all Hilbert spaces and adjoints are understood in
	the real sense. Let \(u_0\in Y\). For
	\(g\in L^2((0,T)\times\T^d;\R)\), the terminal state of the solution of
	\eqref{equ-general-real-control}  is
	\[
		u(T)=U(T)u_0+\Phi_Tg,
	\]
	where \(\Phi_Tg\) is given by \eqref{2.2-wgs-7-6}. Since \(U(t)\) is
	unitary on \(L^2(\T^d;\C)\), we can easily obtain
	\[
	\begin{aligned}
		\|\Phi_Tg\|_{L^2(\T^d)}
		\leq
		T^{1/2}\|g\|_{L^2((0,T)\times\T^d)} .
	\end{aligned}
	\]
	Together with assumption \textup{(2)}, this shows that \(\Phi_T\) is a
	bounded real-linear operator from \(L^2((0,T)\times\T^d;\R)\) into \(Y\).
	By assumption \textup{(1)}, we also have \(U(T)u_0\in Y\).

	We compute the real Hilbert adjoint of \(\Phi_T\). For
	\(g\in L^2((0,T)\times\T^d;\R)\) and \(\varphi\in Y\), Fubini's theorem
	and the adjointness of \(U(T-s)\) give
	\[
	\begin{aligned}
		\langle \Phi_Tg,\varphi\rangle_\R
		=
		\int_0^T\int_{\T^d}
		g(s,x)\mathbf 1_G(s,x)
		\operatorname{Re}(U(T-s)^*\varphi)(x)
		\,\d x\,\d s .
	\end{aligned}
	\]
	Hence, the real Hilbert adjoint
$	\Phi_T^*:
	Y\longrightarrow L^2((0,T)\times\T^d;\R)$
is given by
\begin{equation}\label{eq:real-HUM-adjoint}
	\Phi_T^*\varphi
	=
	\mathbf I_G\bigl(\operatorname{Re}(U(T-t)^*\varphi)\bigr),
	\qquad
	\varphi\in Y .
\end{equation}
	By \eqref{eq:real-HUM-adjoint} and \eqref{equ-HUM-observability},
	\(\Phi_T^*\) is bounded from below on \(Y\):
	\[
		\|\varphi\|_{L^2(\T^d)}^2
		\leq
		C\|\Phi_T^*\varphi\|_{L^2((0,T)\times\T^d;\R)}^2,
		\qquad \varphi\in Y .
	\]

	Define \(\Lambda:=\Phi_T\Phi_T^*:Y\to Y\). Then \(\Lambda\) is continuous
	and coercive on the real Hilbert space \(Y\), since
	\[
		\langle \Lambda\varphi,\varphi\rangle_\R
		=
		\|\Phi_T^*\varphi\|_{L^2((0,T)\times\T^d;\R)}^2
		\geq
		C^{-1}\|\varphi\|_{L^2(\T^d)}^2,
		\qquad \varphi\in Y .
	\]
	By the Lax--Milgram theorem, \(\Lambda\) is onto \(Y\). Therefore there
	exists \(\varphi\in Y\) such that
	\[
		\Lambda\varphi=-U(T)u_0 .
	\]
	Set \(g:=\Phi_T^*\varphi\). Then \(g\in L^2((0,T)\times\T^d;\R)\), and
	the terminal state of the corresponding solution of
	\eqref{equ-general-real-control}  is
	\[
		u(T)
		=
		U(T)u_0+\Phi_Tg
		=
		U(T)u_0+\Phi_T\Phi_T^*\varphi
		=
		U(T)u_0+\Lambda\varphi
		=
		0 .
	\]
	This proves the proposition.
\end{proof}

We next record the stationary obstruction which gives the necessary restriction
on controllable initial data.

\begin{proposition}[Stationary obstruction]
\label{prop:stationary-obstruction}
	Let \(Z\subset \ker H\) be a real subspace such that, for every
	\(g\in L^2((0,T)\times\T^d;\R)\) and every \(z\in Z\),
	\begin{equation}\label{eq:stationary-obstruction-assumption}
		\left\langle
		(\mathbf I_G g)(t),z
		\right\rangle_\R
		=
		0,
		\qquad \text{for a.e. }t\in(0,T).
	\end{equation}
	Then every initial datum which can be controlled to zero at time \(T\) for
	\eqref{equ-general-real-control} belongs to \(Z^{\perp_\R}\).
\end{proposition}

\begin{proof}
	Let \(u_0\in L^2(\T^d;\C)\) be controllable to zero at time \(T\). Then
	there exists \(g\in L^2((0,T)\times\T^d;\R)\) such that the corresponding
	solution \(u\) of \eqref{equ-general-real-control} satisfies \(u(T)=0\). Fix \(z\in Z\). Since \(z\in\ker H\),
	we have that $U(t)^*z=z$ for every $t\in\R$.
	By Duhamel's formula,
	\[
		u(t)
		=
		U(t)u_0
		+
		\int_0^t U(t-s)\bigl((\mathbf I_Gg)(s)\bigr)\,\d s
		\quad\text{in }L^2(\T^d;\C).
	\]
	Taking the real \(L^2\)-pairing with \(z\) and using
	\eqref{eq:stationary-obstruction-assumption}, we get
	\[
	\begin{aligned}
		\langle u(t),z\rangle_\R
		&=
		\langle u_0,z\rangle_\R
		+
		\int_0^t
		\left\langle
		(\mathbf I_Gg)(s),z
		\right\rangle_\R
		\,\d s  \\
		&=
		\langle u_0,z\rangle_\R,
		\qquad t\in[0,T].
	\end{aligned}
	\]
	Taking \(t=T\) and using \(u(T)=0\), we obtain
	\(\langle u_0,z\rangle_\R=0\). Since \(z\in Z\) was arbitrary,
	\(u_0\in Z^{\perp_\R}\).
\end{proof}

\subsection[Reduction of Theorem 1.1 to the open-set real-part estimate]{Reduction of \texorpdfstring{\thmref{thm-controlspace}}{Theorem 1.1} to the open-set real-part estimate}
\label{subsec:open-reduction}

We first recall that \(\mathcal K_{\operatorname{Im}}\) is the purely
imaginary stationary subspace defined in \eqref{equ-ker-ReIm}, and that
\(P_{\mathcal K_{\operatorname{Im}}}\) denotes the real orthogonal projection
onto \(\mathcal K_{\operatorname{Im}}\). Let
$G=(0,T)\times E$,
where \(E\subset\T^d\) is nonempty and open. The aim of this subsection is to
show that \thmref{thm-controlspace} follows from the following open-set
real-part estimate:
\begin{equation}\label{2.4-7-5-w}
\|\psi\|_{L^2(\T^d)}^2
\leq
C\left\|
\mathbf I_G\bigl(\operatorname{Re}(e^{it(\Delta-V)}\psi)\bigr)
\right\|_{L^2((0,T)\times\T^d)}^2
+
\|P_{\mathcal K_{\operatorname{Im}}}\psi\|_{L^2(\T^d)}^2,
\qquad
\psi\in L^2(\T^d;\C),
\end{equation}
where \(C>0\) is independent of \(\psi\). This estimate will be proved later
in \thmref{thm-obRe-V}.

We now show that \eqref{2.4-7-5-w} implies \thmref{thm-controlspace}. We take
\[
	H=-\Delta+V,
	\qquad
	U(t)=e^{-itH}=e^{it(\Delta-V)},
	\qquad t\in\R .
\]
Let \(X\) be the controllable subspace at time \(T\) for
\eqref{equ-general-real-control}. Set
\[
	Y=\mathcal K_{\operatorname{Im}}^{\perp_\R}.
\]

Since \(Y\) is a real orthogonal complement, it is a closed real
subspace of \(L^2(\T^d;\C)\). We verify the two structural assumptions in Proposition~\ref{prop:HUM-sufficient}. Since \(\mathcal K_{\operatorname{Im}}\subset\ker H\), every
\(z\in\mathcal K_{\operatorname{Im}}\) satisfies
\[
	U(t)z=z,
	\qquad
	U(t)^*z=z,
	\qquad t\in\R .
\]
Thus, if \(y\in Y\) and \(z\in\mathcal K_{\operatorname{Im}}\), then 
\[
	\langle U(t)y,z\rangle_\R
	=
	\langle y,U(t)^*z\rangle_\R
	=
	\langle y,z\rangle_\R
	=
	0,\;\;\mbox{for every}\; t\in\R.
\]
Hence \(Y\) is invariant under \(U(t)\) for every $t\in\R$.

Next, let \(g\in L^2((0,T)\times\T^d;\R)\), and let \(\Phi_Tg\) be given by
\eqref{2.2-wgs-7-6}. For \(z\in\mathcal K_{\operatorname{Im}}\), the
adjointness of \(U(T-s)\) and the identity \(U(T-s)^*z=z\) give
\[
 	\langle \Phi_Tg,z\rangle_\R
	 =
	\int_0^T
	\left\langle
	(\mathbf I_G g)(s),U(T-s)^*z
	\right\rangle_\R
	\,\d s                                                     =
	\int_0^T
	\left\langle
	(\mathbf I_G g)(s),z
	\right\rangle_\R
	\,\d s.
\]
Since \(z\) is purely imaginary and
\((\mathbf I_G g)(s)\) is real-valued for a.e. \(s\in(0,T)\), we have
\[
\left\langle (\mathbf I_G g)(s),z\right\rangle_\R=0
\quad\text{for a.e. }s\in(0,T).
\]
Therefore,
\(\Phi_Tg\in\mathcal K_{\operatorname{Im}}^{\perp_\R}=Y\).

We now prove the observability estimate \eqref{equ-HUM-observability} required in
Proposition~\ref{prop:HUM-sufficient}. Let \(\varphi\in Y\) and set
\(\psi=U(-T)\varphi\). Since \(Y\) is invariant under \(U(t)\), we have
\(\psi\in Y\), and hence \(P_{\mathcal K_{\operatorname{Im}}}\psi=0\).
Applying \eqref{2.4-7-5-w} to \(\psi\), and using the unitarity of \(U(t)\),
we obtain
\[
\begin{aligned}
\|\varphi\|_{L^2(\T^d)}^2
&=
\|\psi\|_{L^2(\T^d)}^2                                                     \\
&\leq
C\left\|
\mathbf I_G\bigl(\operatorname{Re}(U(t)U(-T)\varphi)\bigr)
\right\|_{L^2((0,T)\times\T^d)}^2                                           \\
&=
C\left\|
\mathbf I_G\bigl(\operatorname{Re}(U(T-t)^*\varphi)\bigr)
\right\|_{L^2((0,T)\times\T^d)}^2,
\end{aligned}
\]
which leads to  \eqref{equ-HUM-observability}. Proposition~\ref{prop:HUM-sufficient}
therefore yields
$	\mathcal K_{\operatorname{Im}}^{\perp_\R}=Y\subset X$.

For the reverse inclusion, we apply Proposition~\ref{prop:stationary-obstruction}
with \(Z=\mathcal K_{\operatorname{Im}}\). Let
\(g\in L^2((0,T)\times\T^d;\R)\) and
\(z\in\mathcal K_{\operatorname{Im}}\). Then \((\mathbf I_G g)(t,x)\) is real
for a.e. \((t,x)\in(0,T)\times\T^d\), while \(z(x)\) is purely imaginary for
a.e. \(x\in\T^d\). Hence
\[
	\operatorname{Re}
	\bigl((\mathbf I_G g)(t,x)\overline{z(x)}\bigr)
	=
	0,
	\qquad
	\text{for a.e. }(t,x)\in(0,T)\times\T^d .
\]
Since
\[
	(\mathbf I_G g)(t,x)\overline{z(x)}
	\in L^1((0,T)\times\T^d),
\]
Fubini's theorem gives
\[
	\operatorname{Re}\int_{\T^d}
	(\mathbf I_G g)(t,x)\overline{z(x)}\,\d x
	=
	0,
	\qquad \text{for a.e. }t\in(0,T).
\]
Thus Proposition~\ref{prop:stationary-obstruction} gives the reverse
inclusion \(X\subset\mathcal K_{\operatorname{Im}}^{\perp_\R}\).

Combining the two inclusions, we obtain
$X=\mathcal K_{\operatorname{Im}}^{\perp_\R}$,
which proves \thmref{thm-controlspace}.

The two special cases mentioned after \thmref{thm-controlspace} follow
immediately. If \(V\geq 0\) and \(\int_{\T^d}V>0\), then
\(\mathcal K=\{0\}\) by Lemma \ref{lem-positive-V}. If \(V=0\), then
\(\mathcal K\) consists of the constant functions, and
\(\mathcal K_{\operatorname{Im}}\) consists of the purely imaginary constants.
Therefore
\[
	X
	=
	\left\{
	\phi\in L^2(\T^d;\C):
	\operatorname{Im}\int_{\T^d}\phi(x)\,\d x=0
	\right\}.
\]

\subsection[Reduction of Theorem 1.2 to the real-part estimate]{Reduction of \texorpdfstring{\thmref{thm:measurable-control}}{Theorem 1.2} to the real-part estimate}
\label{subsec:measurable-reduction}

In this subsection we show that \thmref{thm:measurable-control} follows from
the real-part estimate for admissible measurable product sets. We recall that
\((\mathcal K_c)_{\operatorname{Im}}\) is the purely imaginary stationary
subspace associated with the shifted free operator, and that
\(P_{(\mathcal K_c)_{\operatorname{Im}}}\) denotes the real orthogonal
projection onto \((\mathcal K_c)_{\operatorname{Im}}\).

For \thmref{thm:measurable-control}, we work with the shifted free equation
\eqref{equ-shifted-control} on
$Q=(0,2\pi)\times\T^d$,
where \(G\subset Q\) is an admissible measurable product set. In this
subsection we use the same notation \(\mathbf I_G\), with \(T=2\pi\), for the
multiplication operator on \(L^2(Q;\mathbb F)\), \(\mathbb F\in\{\R,\C\}\),
defined in \eqref{eq:def-IG}. We take
\[
	H_c=-\Delta-c,
	\qquad
	U_c(t)=e^{-itH_c}=e^{it(\Delta+c)},\qquad t\in\R .
\]
The estimate to be proved later is
\begin{equation}\label{eq:measurable-real-part-estimate}
\|\phi\|_{L^2(\T^d)}^2
\leq
C\left\|
\mathbf I_G\bigl(\operatorname{Re}(U_c(t)\phi)\bigr)
\right\|_{L^2(Q)}^2
+
\|P_{(\mathcal K_c)_{\operatorname{Im}}}\phi\|_{L^2(\T^d)}^2,
\qquad
\phi\in L^2(\T^d;\C),
\end{equation}
where \(C>0\) is independent of \(\phi\). The stationary eigenspace is
\[
\mathcal K_c
=
\ker H_c
=
\operatorname{span}\{e^{i\langle n,x\rangle}: n\in\Z^d,\ |n|^2=c\}.
\]

Let \(X\) be the controllable subspace at time \(2\pi\) for
\eqref{equ-shifted-control}, and set
\[
	Y_c=\bigl((\mathcal K_c)_{\operatorname{Im}}\bigr)^{\perp_\R}.
\]
We verify the two structural assumptions in
Proposition~\ref{prop:HUM-sufficient} with
\[
	T=2\pi,\qquad Y=Y_c,\qquad H=H_c,\qquad U=U_c .
\]

Since \((\mathcal K_c)_{\operatorname{Im}}\subset\ker H_c\), every element of
\((\mathcal K_c)_{\operatorname{Im}}\) is fixed by both \(U_c(t)\) and
\(U_c(t)^*\). Hence, if \(y\in Y_c\) and
\(z\in(\mathcal K_c)_{\operatorname{Im}}\), then
\[
	\langle U_c(t)y,z\rangle_\R
	=
	\langle y,U_c(t)^*z\rangle_\R
	=
	\langle y,z\rangle_\R
	=
	0,
	\qquad t\in\R .
\]
Thus \(Y_c\) is invariant under \(U_c(t)\).

Next, let \(\Phi_{2\pi}^c\) be the operator given by
\eqref{2.2-wgs-7-6} with \(T=2\pi\) and \(U=U_c\). We show that
\[
	\operatorname{Ran}\Phi_{2\pi}^c\subset Y_c .
\]
Let \(g\in L^2(Q;\R)\) and \(z\in(\mathcal K_c)_{\operatorname{Im}}\). Since
\(U_c(2\pi-s)^*z=z\), we have
\[
\begin{aligned}
	\langle \Phi_{2\pi}^c g,z\rangle_\R
	&=
	\int_0^{2\pi}
	\left\langle
	U_c(2\pi-s)\bigl((\mathbf I_G g)(s)\bigr),z
	\right\rangle_\R
	\,\d s                                                     \\
	&=
	\int_0^{2\pi}
	\left\langle
	(\mathbf I_G g)(s),U_c(2\pi-s)^*z
	\right\rangle_\R
	\,\d s                                                     \\
	&=
	\int_0^{2\pi}
	\left\langle
	(\mathbf I_G g)(s),z
	\right\rangle_\R
	\,\d s
	=
	0 .
\end{aligned}
\]
The last equality follows because \(z\) is purely imaginary, whereas
\((\mathbf I_Gg)(s)\) is real-valued for a.e. \(s\in(0,2\pi)\).
Hence \(\Phi_{2\pi}^c g\in Y_c\). This verifies the second structural
assumption in Proposition~\ref{prop:HUM-sufficient}.

We now derive the observability estimate. Let \(\varphi\in Y_c\) and set
$	\psi=U_c(-2\pi)\varphi$. 
Since \(Y_c\) is invariant under \(U_c(t)\), we have \(\psi\in Y_c\). Hence, $P_{(\mathcal K_c)_{\operatorname{Im}}}\psi=0$.
Applying \eqref{eq:measurable-real-part-estimate} to this \(\psi\), and using
the unitarity of \(U_c(t)\), we get
\[
\begin{aligned}
\|\varphi\|_{L^2(\T^d)}^2
&=
\|\psi\|_{L^2(\T^d)}^2                                                    \\
&\leq
C\left\|
\mathbf I_G\bigl(\operatorname{Re}(U_c(t)\psi)\bigr)
\right\|_{L^2(Q)}^2                                                        \\
&=
C\left\|
\mathbf I_G\bigl(\operatorname{Re}(U_c(t)U_c(-2\pi)\varphi)\bigr)
\right\|_{L^2(Q)}^2                                                        \\
&=
C\left\|
\mathbf I_G\bigl(\operatorname{Re}(U_c(2\pi-t)^*\varphi)\bigr)
\right\|_{L^2(Q)}^2 .
\end{aligned}
\]
This is precisely \eqref{equ-HUM-observability} in
Proposition~\ref{prop:HUM-sufficient}, with \(T=2\pi\), \(U=U_c\), and
\(Y=Y_c\). Therefore Proposition~\ref{prop:HUM-sufficient} gives
$Y_c\subset X$. 

It remains to prove the reverse inclusion. We apply
Proposition~\ref{prop:stationary-obstruction} with
\[
	T=2\pi,
	\qquad
	H=H_c,
	\qquad
	Z=(\mathcal K_c)_{\operatorname{Im}} .
\]
Indeed, \(Z\subset\ker H_c\). Moreover, for every \(g\in L^2(Q;\R)\) and
every \(z\in(\mathcal K_c)_{\operatorname{Im}}\), we have, for a.e.
\(t\in(0,2\pi)\),
\[
	(\mathbf I_G g)(t)\in L^2(\T^d;\R),
	\qquad
	z\in L^2(\T^d;i\R).
\]
Therefore
\[
	\operatorname{Re}\int_{\T^d}
	(\mathbf I_G g)(t,x)\overline{z(x)}\,\d x
	=
	0,
	\qquad \text{for a.e. }t\in(0,2\pi).
\]
Thus the hypothesis of Proposition~\ref{prop:stationary-obstruction} is
satisfied, and therefore
$X\subset Y_c$.
Combining the two inclusions, we obtain
$	X
	=
	\bigl((\mathcal K_c)_{\operatorname{Im}}\bigr)^{\perp_\R}$,
which proves \thmref{thm:measurable-control}.

The full-space controllability criterion follows immediately. The controllable
subspace is the whole \(L^2(\T^d;\C)\) if and only if
$	(\mathcal K_c)_{\operatorname{Im}}=\{0\}$.
For the shifted free operator, this is equivalent to
\[
	c\notin\{|n|^2:n\in\Z^d\}.
\]
Indeed, if \(c=|n|^2\) for some \(n\in\Z^d\), then
\((\mathcal K_c)_{\operatorname{Im}}\neq\{0\}\): when \(n=0\), the function
\(i\) belongs to \((\mathcal K_c)_{\operatorname{Im}}\), and when \(n\neq0\),
the function
\[
	i\cos\langle n,x\rangle
\]
belongs to \((\mathcal K_c)_{\operatorname{Im}}\). Conversely, if no such
\(n\) exists, then \(\mathcal K_c=\{0\}\).

\section{Real-part estimate from open sets}
\label{sec:open-real-part}
Throughout this section, we fix $T>0$, assume that
$V\in C(\mathbb T^d;\mathbb R)$, and let $E\subset\mathbb T^d$ be
nonempty and open. We consider the homogeneous Schr\"odinger equation
\begin{equation}\label{equ-sch-V}
	\left\{
	\begin{aligned}
		i\partial_t u&=(-\Delta+V)u,
		&& (t,x)\in (0,T)\times\mathbb T^d,\\
		u|_{t=0}&=u_0,
	\end{aligned}
	\right.
\end{equation}
where \(u_0\in L^2(\mathbb T^d;\mathbb C)\).
 We use the notation
\(\mathcal K\), \(\mathcal K_{\operatorname{Im}}\),
\(\mathcal K_{\operatorname{Im}}^{\perp_{\mathbb R}}\), and
\(P_{\mathcal K_{\operatorname{Im}}}\) introduced in
\eqref{1.3-7-14-w}, \eqref{equ-ker-ReIm}, \eqref{1.5-7-14-w},
and \eqref{eq:P-Kim-definition}, respectively.

The aim of this section is to prove the following open-set real-part estimate:

\begin{theorem}[Real-part estimate from open sets]\label{thm-obRe-V}
	There exists a constant \(C=C(V,d,T,E)>0\) such that, for every
	\(u_0\in L^2(\mathbb T^d;\mathbb C)\), the corresponding solution of
	\eqref{equ-sch-V} satisfies
	\begin{equation}\label{equ-thm1-1}
		\|u_0\|^2_{L^2(\mathbb T^d)}
		\le
		C\int_0^T\int_E |\operatorname{Re}u(t,x)|^2\,\d x\,\d t
		+
		\|P_{\mathcal K_{\operatorname{Im}}}u_0\|^2_{L^2(\mathbb T^d)} .
	\end{equation}
	Moreover, the correction term is optimal in the sense that it cannot be
replaced by the projection onto any proper real subspace of
\(\mathcal K_{\operatorname{Im}}\).
\end{theorem}

\begin{corollary}\label{cor-positive-potential}
	Under the assumptions above, suppose that \(V\ge0\) and
	\[
	\int_{\mathbb T^d}V(x)\,\d x>0.
	\]
	Then there exists \(C=C(V,d,T,E)>0\) such that, for every solution \(u\)
	of \eqref{equ-sch-V},
	\begin{align}\label{equ-Re-ob-positive-potential}	
	\|u_0\|^2_{L^2(\mathbb T^d)}
	\le
	C\int_0^T\int_E |\operatorname{Re}u(t,x)|^2\,\d x\,\d t .
	\end{align}
\end{corollary}

\begin{proof}
	By Lemma~\ref{lem-positive-V} in the Appendix, the assumptions
	\(V\ge0\) and
	\[
		\int_{\mathbb T^d}V(x)\,\d x>0
	\]
	imply that \(\mathcal K=\{0\}\). Hence
	\(\mathcal K_{\operatorname{Im}}=\{0\}\). Therefore the correction term
	in \eqref{equ-thm1-1} vanishes, and
	\eqref{equ-Re-ob-positive-potential} follows directly from
	Theorem~\ref{thm-obRe-V}.
\end{proof}

\begin{remark}
Writing
\[
u=y+iz,
\qquad
y,z:\,(0,T)\times\mathbb T^d\longrightarrow\mathbb R,
\]
equation \eqref{equ-sch-V} is equivalent to the coupled
real Hamiltonian system
\begin{equation}\label{equ-coupled-real-system}
	\left\{
	\begin{aligned}
		\partial_t y &=(-\Delta+V)z,\\
		\partial_t z &=-(-\Delta+V)y,\\
		(y,z)|_{t=0}&=(y_0,z_0).
	\end{aligned}
	\right.
\end{equation}
Under the assumptions of \corref{cor-positive-potential}, the preceding
estimate \eqref{equ-Re-ob-positive-potential} is equivalent to
\[
\|y_0\|_{L^2(\mathbb T^d)}^2
+
\|z_0\|_{L^2(\mathbb T^d)}^2
\le
C\int_0^T\int_E |y(t,x)|^2\,\d x\,\d t
\]
for every solution $(y,z)$ of \eqref{equ-coupled-real-system}. Thus, in this
Hamiltonian system, observing only one component $y$ on $(0,T)\times E$
stably determines both initial components. This should be compared
with one-component or partial observability results for coupled parabolic
systems, including the Kalman-type theory and related parabolic systems with
one control force
\cite{AmmarKhodjaBenabdallahGonzalezTeresa2011Survey,
	BenabdallahCannarsaYamamoto2009,Guerrero2007,FernandezCaraGonzalezTeresa2010,
	LissyZuazua2019}, and with the recent measurable-set result for strongly
coupled parabolic systems of Fu--Wang--Yu--Zhu \cite{FuWangYu26}. We also
mention the coupled dispersive Schr\"odinger--KdV system of
Araruna--Cerpa--Mercado--Santos \cite{ArarunaCerpaMercadoSantos2016}, where
Carleman estimates are used to prove internal null controllability for a linear
Schr\"odinger--KdV system on a bounded interval.
	
\end{remark}

We now explain the organization of the proof of Theorem~\ref{thm-obRe-V}.
The proof has three steps.  In \subsecref{subsec:open-compact-remainder}, we prove
a real-part observability estimate up to a compact remainder.  The key input is
a high-frequency real-part estimate; the low-frequency part is kept as a
compact term.  In \subsecref{subsec:ucp-observable-subspace}, we identify the
invisible space and use this identification to obtain an observability estimate
on \(\mathcal K_{\operatorname{Im}}^{\perp_{\mathbb R}}\).  Finally, in
\subsecref{subsec:proof-open-real-part-estimate}, we combine this estimate with
the real orthogonal decomposition
\[
L^2(\mathbb T^d;\mathbb C)
=
\mathcal K_{\operatorname{Im}}^{\perp_{\mathbb R}}
\oplus
\mathcal K_{\operatorname{Im}}
\]
to prove Theorem~\ref{thm-obRe-V}.

\subsection{Real-part observability up to a compact remainder}
\label{subsec:open-compact-remainder}

The main result of this subsection is an observability estimate with a compact
remainder. Its proof is based on a high-frequency real-part estimate, which is
obtained by combining the classical complex-valued observability inequality
with an oscillatory cancellation argument for the interaction between the
forward and backward Schr\"odinger branches.

We first introduce the spectral notation for \(H=-\Delta+V\) and the
associated Sobolev scale. Since \(V\) is real-valued and continuous, the
operator \(H\) is self-adjoint on \(L^2(\mathbb T^d;\mathbb C)\), with domain
\(H^2(\mathbb T^d)\). Its spectrum is discrete. Let
\[
	\lambda_1\le\lambda_2\le\cdots\le\lambda_n\le\cdots\to+\infty
\]
be the eigenvalues of \(H\), counted with multiplicity, and let
\(\{\varphi_n\}_{n\ge1}\) be an orthonormal basis of
\(L^2(\mathbb T^d;\mathbb C)\) consisting of eigenfunctions of \(H\),
\[
	H\varphi_n=\lambda_n\varphi_n .
\]
 For
every real \(s\), we denote by \(\mathcal H^s\) the Hilbert space of functions
with finite norm
\[
	\|u\|^2_{\mathcal H^s}
	=
	\sum_n(1+|\lambda_n|^2)^s |(u,\varphi_n)|^2.
\]

\begin{proposition}[Observability up to a compact term]\label{prop-compact-remainder}
	There exists a constant \(C=C(V,d,T,E)>0\) such that, for every
	\(u_0\in L^2(\mathbb T^d;\mathbb C)\),
	\begin{equation}\label{equ-314-1}
		\|u_0\|^2_{L^2(\mathbb T^d)}
		\le
		C\int_0^T
		\|\operatorname{Re}e^{it(\Delta-V)}u_0\|^2_{L^2(E)}\,\d t
		+
		C\|u_0\|^2_{\mathcal H^{-2}}.
	\end{equation}
\end{proposition}

The rest of this subsection is devoted to the proof of
\propref{prop-compact-remainder}. We first recall the classical observability
inequality for the Schr\"odinger equation on the torus, due to Anantharaman
and Macià; see \cite{AnaJEMS2014}.

\begin{proposition}[Classical observability]\label{prop-ob-class}
For every \(\tau>0\), there exists a constant
	\(C_0=C_0(V,d,\tau,E)>0\) such that
	\[
	C_0\|u_0\|^2_{L^2(\mathbb T^d)}
	\le
	\int_0^\tau
	\|e^{it(\Delta-V)}u_0\|_{L^2(E)}^2\,\d t,
	\qquad
	u_0\in L^2(\mathbb T^d;\mathbb C).
	\]
\end{proposition}

We now introduce the remaining spectral notation needed in the proof. For
\(N\in\mathbb R\), let \(\Pi_N\) be the orthogonal projection onto the
finite-dimensional subspace spanned by those eigenfunctions \(\varphi_n\) with
\(\lambda_n\le N\). Thus \(\Pi_Nu\) is the low-frequency part of \(u\), while
\((I-\Pi_N)u\) is the corresponding high-frequency part. We do not assume any
sign condition on \(V\), and therefore some low eigenvalues may be negative.

We shall use the following dyadic consequence of Weyl's formula for the
eigenvalue counting function; see, for instance, \cite{HuangSoggeCPDEWeyllaw}.
There exist constants \(C_W>0\) and \(L_0>0\), depending only on \(V\) and
\(d\), such that, for every \(L\ge L_0\),
\begin{equation}\label{eq:weyl-dyadic}
	\#\{n:L<\lambda_n\le2L\}
	\le
	C_W(2L)^{d/2}.
\end{equation}

We also record a time-localized form of the classical observability estimate.
Applying Proposition~\ref{prop-ob-class} with \(\tau=T/6\) to
\(e^{-i(T/3)H}u_0\), and using the unitarity and group properties of
\(e^{-itH}\), we obtain
\begin{equation}\label{equ-312-1}
	C_1\|u_0\|^2_{L^2(\mathbb T^d)}
	\le
	\int_{T/3}^{T/2}
	\|e^{it(\Delta-V)}u_0\|^2_{L^2(E)}\,\d t,
	\qquad
	\forall u_0\in L^2(\mathbb T^d;\mathbb C),
\end{equation}
where \(C_1=C_0(V,d,T/6,E)>0\).

\begin{lemma}[High-frequency real-part estimate]\label{lem-high-V}
	Let \(C_1>0\) be the constant in \eqref{equ-312-1}. Then there exists
	\(N_0>0\), depending only on \(V,d,T,E\), such that, for all
	\(N\ge N_0\) and all \(u_0\in L^2(\mathbb T^d;\mathbb C)\),
	\begin{equation}\label{equ-312-2}
		\frac{C_1}{4}\|(I-\Pi_N)u_0\|^2_{L^2(\mathbb T^d)}
		\le
		\int_0^T
		\|\operatorname{Re}e^{it(\Delta-V)}(I-\Pi_N)u_0\|^2_{L^2(E)}\,\d t.
	\end{equation}
\end{lemma}

\begin{proof}
		Let \(\chi\) be a smooth nonnegative cutoff function on \(\mathbb R\) such
	that \(0\leq\chi\leq1\), \(\chi=1\) on \([1/3,1/2]\), and
	\(\operatorname{supp}\chi\subset[0,1]\). Set
	\(\chi_T(t)=\chi(t/T)\). Then
	\(\chi_T=1\) on \([T/3,T/2]\) and
	\(\operatorname{supp}\chi_T\subset[0,T]\).

	Since \(H=-\Delta+V\) has real coefficients, we have
	\[
		\operatorname{Re}e^{it(\Delta-V)}\varphi
		=
		\frac12
		\left(
		e^{it(\Delta-V)}\varphi
		+
		e^{-it(\Delta-V)}\overline{\varphi}
		\right).
	\]
	Therefore,
	\begin{equation}\label{equ-312-3}
		\begin{aligned}
			\int_0^T
			\|\operatorname{Re}e^{it(\Delta-V)}\varphi\|_{L^2(E)}^2\,\d t
			&\ge
			\int_{\mathbb R}\chi_T(t)
			\|\operatorname{Re}e^{it(\Delta-V)}\varphi\|_{L^2(E)}^2\,\d t\\
			&=
			\frac12
			\int_{\mathbb R}\chi_T(t)
			\|e^{it(\Delta-V)}\varphi\|_{L^2(E)}^2\,\d t\\
			&\quad+
			\frac14
			\int_{\mathbb R}\chi_T(t)\int_E
			\left[
			\left(e^{it(\Delta-V)}\varphi\right)^2
			+
			\left(e^{-it(\Delta-V)}\overline{\varphi}\right)^2
			\right]\,\d x\,\d t .
		\end{aligned}
	\end{equation}

		We now estimate the two terms on the right-hand side of
	\eqref{equ-312-3}. First, since \(\chi_T=1\) on \([T/3,T/2]\), the
	estimate \eqref{equ-312-1} gives
	\begin{equation}\label{equ-312-3.5}
		\frac12
		\int_{\mathbb R}\chi_T(t)
		\|e^{it(\Delta-V)}\varphi\|_{L^2(E)}^2\,\d t
		\ge
		\frac{C_1}{2}
		\|\varphi\|_{L^2(\mathbb T^d)}^2 .
	\end{equation}
		Second, we claim that there exists \(N_0>0\) such that, for every
	\(N\ge N_0\) and every
	\(\varphi\in\operatorname{Ran}(I-\Pi_N)\),
	\begin{equation}\label{equ-312-4}
		\left|
		\int_{\mathbb R}\chi_T(t)\int_E
		\left[
		\left(e^{it(\Delta-V)}\varphi\right)^2
		+
		\left(e^{-it(\Delta-V)}\overline{\varphi}\right)^2
		\right]\,\d x\,\d t
		\right|
		\le
		C_1\|\varphi\|_{L^2(\mathbb T^d)}^2 .
	\end{equation}
	Assume for the moment that \eqref{equ-312-4} holds. Taking
	\(\varphi=(I-\Pi_N)u_0\) in \eqref{equ-312-3}, and using
	\eqref{equ-312-3.5} and \eqref{equ-312-4}, we obtain
	\[
	\begin{aligned}
		\int_0^T
		\|\operatorname{Re}e^{it(\Delta-V)}(I-\Pi_N)u_0\|_{L^2(E)}^2\,\d t
		\geq
		\frac{C_1}{4}
		\|(I-\Pi_N)u_0\|_{L^2(\mathbb T^d)}^2 .
	\end{aligned}
	\]
	This is exactly \eqref{equ-312-2}.

It remains to prove \eqref{equ-312-4}. To this end, we first prove it for
high-frequency finite spectral sums and then extend the result to the general
case by approximation.

Let \(N\) be sufficiently large, to be fixed below. Let \(F\) be a finite
set of indices such that \(\lambda_n>N\) for every \(n\in F\), and let
\[
	\varphi=\sum_{n\in F}c_n\varphi_n .
\] Then
\[
	e^{it(\Delta-V)}\varphi
	=
	\sum_{n\in F}c_ne^{-i\lambda_n t}\varphi_n .
\]
Define
\[
	I_+(\varphi)
	:=
	\int_{\mathbb R}\chi_T(t)\int_E
	\left(e^{it(\Delta-V)}\varphi\right)^2\,\d x\,\d t .
\]
	Expanding the square and using Fubini's theorem, we obtain
\begin{equation}\label{equ-Iplus-expansion}
	I_+(\varphi)
	=
	\sum_{m,n\in F}
	\widehat{\chi_T}(\lambda_m+\lambda_n)c_mc_n
	\int_E\varphi_m\varphi_n\,\d x .
\end{equation}
Since
\[
	|\widehat{\chi_T}(\xi)|
	\leq
	C_{k,T}(1+T|\xi|)^{-k},
	\qquad k\geq1,
\]
by taking absolute values in \eqref{equ-Iplus-expansion}, and using
\[
	\left|\int_E\varphi_m\varphi_n\,\d x\right|
	\leq
	\|\varphi_m\|_{L^2(E)}
	\|\varphi_n\|_{L^2(E)}
	\leq
	1,
\]
we get
\[
	|I_+(\varphi)|
	\leq
	C_{k,T}
	\sum_{m,n\in F}
	(1+T|\lambda_m+\lambda_n|)^{-k}|c_m||c_n|.
\]
This, along with \(2|c_m||c_n|\leq |c_m|^2+|c_n|^2\), implies
\begin{equation}\label{3.13-7-11-w}
\begin{aligned}
	|I_+(\varphi)|
	&\leq
	\frac{C_{k,T}}2
	\sum_{m,n\in F}
	(1+T|\lambda_m+\lambda_n|)^{-k}
	\left(|c_m|^2+|c_n|^2\right)\\
	&\leq
	C_{k,T}K_N\sum_{n\in F}|c_n|^2 ,
\end{aligned}
\end{equation}
where
\begin{equation}\label{3.14-7-11-wgs}
	K_N
	=
	\sup_{\lambda_n>N}
	\sum_{\lambda_m>N}
	(1+T|\lambda_m+\lambda_n|)^{-k}.
\end{equation}

	Since \(\{\varphi_n\}\) is an orthonormal basis, we have $	\sum_{n\in F}|c_n|^2
		=
		\|\varphi\|_{L^2(\mathbb T^d)}^2$.
	Along with \eqref{3.13-7-11-w}, this implies
		\begin{equation}\label{equ-Iplus-KN}
		|I_+(\varphi)|
		\le
		C_{k,T}K_N
		\|\varphi\|_{L^2(\mathbb T^d)}^2.
	\end{equation}
	
	We now estimate \(K_N\). Recall that \(L_0\) is the constant appearing in
	the dyadic Weyl estimate \eqref{eq:weyl-dyadic}. Choose
	\(N_0\ge\max\{L_0,1\}\). For \(N\ge N_0\), fix \(n\) with
	\(\lambda_n>N\). Decompose the possible \(\lambda_m\)'s into dyadic shells
	\[
		2^jN<\lambda_m\le2^{j+1}N,\qquad j\ge0.
	\]
	If \(\lambda_n>N\) and \(2^jN<\lambda_m\le2^{j+1}N\), then  $\lambda_m+\lambda_n\ge2^jN$.
	Therefore,
	\[
		(1+T|\lambda_m+\lambda_n|)^{-k}
		\le
		C_T(2^jN)^{-k}.
	\]
	The above, together with \eqref{eq:weyl-dyadic}, yields
\begin{equation}\label{equ-KN-estimate}
		\begin{aligned}
			\sum_{\lambda_m>N}
			(1+T|\lambda_m+\lambda_n|)^{-k}
			&\le
			C_T
			\sum_{j\ge0}
			\#\{m:2^jN<\lambda_m\le2^{j+1}N\}
			(2^jN)^{-k} \\
			&\le
			C_T
			\sum_{j\ge0}
			(2^jN)^{d/2}(2^jN)^{-k}.
		\end{aligned}
	\end{equation}
	If \(k>d/2\), the series in \eqref{equ-KN-estimate} converges. Hence, by
\eqref{3.14-7-11-wgs} and \eqref{equ-KN-estimate},
\[
	K_N\leq C_TN^{d/2-k}.
\]
Taking \(k>d/2+1\), we have \(K_N\to0\) as \(N\to\infty\). Hence, after
increasing \(N_0\) if necessary,
\[
	C_{k,T}K_N\leq\frac{C_1}{2}.
\]
Combining this with \eqref{equ-Iplus-KN}, we obtain
\begin{equation}\label{equ-Iplus-final}
	|I_+(\varphi)|
	\leq
	\frac{C_1}{2}\|\varphi\|_{L^2(\mathbb T^d)}^2.
\end{equation}

By the same argument, for
\[
I_-(\varphi)
:=
\int_{\mathbb R}\chi_T(t)\int_E
\left(e^{-it(\Delta-V)}\overline{\varphi}\right)^2
\,\d x\,\d t ,
\]
we have
\begin{equation}\label{equ-Iminus-final}
		|I_-(\varphi)|
		\le
		\frac{C_1}{2}\|\varphi\|_{L^2(\mathbb T^d)}^2.
	\end{equation}
	Adding \eqref{equ-Iplus-final} and \eqref{equ-Iminus-final}, we obtain
	\eqref{equ-312-4} for high-frequency finite spectral sums.

It remains to pass to general
\(\varphi\in\operatorname{Ran}(I-\Pi_N)\). Choose finite spectral
truncations \(\varphi_M\in\operatorname{Ran}(I-\Pi_N)\) such that
\[
	\varphi_M\to\varphi
	\quad\text{in }L^2(\mathbb T^d;\mathbb C),
	\qquad\text{as }M\to\infty .
\]
By \eqref{equ-Iplus-final} and \eqref{equ-Iminus-final},
\eqref{equ-312-4} holds for each \(\varphi_M\).

Set
\[
	a_M(t,x)=
	(e^{it(\Delta-V)}\varphi_M)(x),
	\qquad
	a(t,x)=
	(e^{it(\Delta-V)}\varphi)(x).
\]
By the unitarity of \(e^{it(\Delta-V)}\),
\begin{equation}\label{eq:aM-a-convergence}
\|a_M-a\|_{L^2((0,T)\times E)}^2
\leq
T\|\varphi_M-\varphi\|_{L^2(\mathbb T^d)}^2
\to0,
\qquad\text{as }M\to\infty .
\end{equation}
Since \(\chi_T\in L^\infty(\mathbb R)\) and
\(\operatorname{supp}\chi_T\subset[0,T]\), we have
\begin{equation}\label{eq:aM-square-limit}
\begin{aligned}
&\left|
\int_{\mathbb R}\chi_T(t)\int_E
(a_M(t,x)^2-a(t,x)^2)\,\d x\,\d t
\right|\\
&\leq
\|\chi_T\|_{L^\infty}
\|a_M-a\|_{L^2((0,T)\times E)}
\|a_M+a\|_{L^2((0,T)\times E)}
\to0, \qquad\text{as }M\to\infty.
\end{aligned}
\end{equation}

The same argument applied to
\[
b_M(t,x)=
(e^{-it(\Delta-V)}\overline{\varphi_M})(x),
\qquad
b(t,x)=
(e^{-it(\Delta-V)}\overline{\varphi})(x)
\]
gives
\begin{equation}\label{eq:bM-square-limit}
\int_{\mathbb R}\chi_T(t)\int_E
(b_M(t,x)^2-b(t,x)^2)\,\d x\,\d t
\to0,
\qquad\text{as }M\to\infty .
\end{equation}

Therefore, by \eqref{eq:aM-square-limit} and
\eqref{eq:bM-square-limit}, we may pass to the limit in
\eqref{equ-312-4}. Hence \eqref{equ-312-4} holds for every
\(\varphi\in\operatorname{Ran}(I-\Pi_N)\). Consequently,
\eqref{equ-312-2} follows from \eqref{equ-312-3},
\eqref{equ-312-3.5}, and \eqref{equ-312-4}.
\end{proof}

We now prove \propref{prop-compact-remainder}.

\begin{proof}[Proof of \propref{prop-compact-remainder}]
	Let \(N=N_0\) be given by \lemref{lem-high-V}. For
	\(u_0\in L^2(\mathbb T^d;\mathbb C)\), we decompose
	\[
		u_0=\Pi_Nu_0+(I-\Pi_N)u_0,
	\]
	which leads to
	\[
		\|u_0\|^2_{L^2(\mathbb T^d)}
		=
		\|\Pi_Nu_0\|^2_{L^2(\mathbb T^d)}
		+
		\|(I-\Pi_N)u_0\|^2_{L^2(\mathbb T^d)} .
	\]
	By \lemref{lem-high-V}, we obtain
	\begin{equation}\label{equ-314-high-frequency}
		\|(I-\Pi_N)u_0\|^2_{L^2(\mathbb T^d)}
		\leq
		\frac4{C_1}
		\int_0^T
		\|\operatorname{Re}e^{it(\Delta-V)}(I-\Pi_N)u_0\|^2_{L^2(E)}
		\,\d t .
	\end{equation}
	Therefore,
	\[
	\|u_0\|^2_{L^2(\mathbb T^d)}
	\leq
	\|\Pi_Nu_0\|^2_{L^2(\mathbb T^d)}
	+
	\frac4{C_1}
	\int_0^T
	\|\operatorname{Re}e^{it(\Delta-V)}(I-\Pi_N)u_0\|^2_{L^2(E)}
	\,\d t .
	\]

	Using
	\[
		\operatorname{Re}e^{it(\Delta-V)}(I-\Pi_N)u_0
		=
		\operatorname{Re}e^{it(\Delta-V)}u_0
		-
		\operatorname{Re}e^{it(\Delta-V)}\Pi_Nu_0
	\]
	and the inequality
	\(\|a-b\|^2\leq2\|a\|^2+2\|b\|^2\), we have
	\begin{equation}\label{equ-314-4}
	\begin{aligned}
		&\int_0^T
		\|\operatorname{Re}e^{it(\Delta-V)}(I-\Pi_N)u_0\|^2_{L^2(E)}
		\,\d t\\
		&\leq
		2\int_0^T
		\|\operatorname{Re}e^{it(\Delta-V)}u_0\|^2_{L^2(E)}
		\,\d t +
		2\int_0^T
		\|\operatorname{Re}e^{it(\Delta-V)}\Pi_Nu_0\|^2_{L^2(E)}
		\,\d t\\
		&\leq 	2\int_0^T
		\|\operatorname{Re}e^{it(\Delta-V)}u_0\|^2_{L^2(E)}
		\,\d t+2T\|\Pi_Nu_0\|^2_{L^2(\mathbb T^d)}.
	\end{aligned}
	\end{equation}
Here we used the fact
	\[
		\|\operatorname{Re}e^{it(\Delta-V)}\Pi_Nu_0\|_{L^2(E)}
		\leq
		\|e^{it(\Delta-V)}\Pi_Nu_0\|_{L^2(\mathbb T^d)}
		=
		\|\Pi_Nu_0\|_{L^2(\mathbb T^d)} .
	\]
	Combining \eqref{equ-314-high-frequency} and
	\eqref{equ-314-4}, we obtain
	\begin{equation}\label{equ-314-5}
		\|u_0\|^2_{L^2(\mathbb T^d)}
		\leq
		C\int_0^T
		\|\operatorname{Re}e^{it(\Delta-V)}u_0\|^2_{L^2(E)}
		\,\d t
		+
		C\|\Pi_Nu_0\|^2_{L^2(\mathbb T^d)},
	\end{equation}
	where \(C\) depends only on \(V,d,T,E\).

	It remains to estimate the low-frequency term. Since \(N=N_0\) is fixed,
	the range of \(\Pi_N\) is finite-dimensional. Writing
	$		u_0=\sum_n c_n\varphi_n$,
		we see
	$		\Pi_Nu_0=\sum_{\lambda_n\leq N}c_n\varphi_n$.
		Therefore,
	\begin{equation}\label{3.26-7-9-wgs}
		\|\Pi_Nu_0\|_{L^2(\mathbb T^d)}^2
		 =
		\sum_{\lambda_n\leq N}|c_n|^2 \leq
		C_N
		\sum_{\lambda_n\leq N}
		(1+|\lambda_n|^2)^{-2}|c_n|^2 \leq
		C_N\|u_0\|_{\mathcal H^{-2}}^2 .
	\end{equation}
	Here \(C_N<\infty\) since the range of \(\Pi_N\) is finite-dimensional.
	Moreover, \(N=N_0\) depends only on \(V,d,T,E\), and hence so does
	\(C_N\).

	Substituting \eqref{3.26-7-9-wgs} into \eqref{equ-314-5} yields
	\eqref{equ-314-1}. This completes the proof of
	\propref{prop-compact-remainder}.
\end{proof}

\subsection{Unique continuation and the observable subspace estimate}
\label{subsec:ucp-observable-subspace}

Let
\begin{equation}\label{equ-NT-definition}
	\mathcal N_T
	=
	\left\{
	u_0\in L^2(\mathbb T^d;\mathbb C):
	\operatorname{Re}e^{it(\Delta-V)}u_0=0
	\text{ on }(0,T)\times E
	\right\}.
\end{equation}
The two main conclusions of this subsection are the following lemma and
proposition.

\begin{lemma}[Unique continuation for the real-part observation]\label{lem-ucp}
	With the above notation, one has
	\[
	\mathcal N_T=\mathcal K_{\operatorname{Im}}.
	\]
	Equivalently, the only initial data whose real-part observation vanishes on
	\((0,T)\times E\) are the purely imaginary stationary modes.
\end{lemma}

\begin{proposition}[Real-part estimate on the observable subspace]
	\label{prop-ob-M}
	There exists a constant \(C=C(V,d,T,E)>0\) such that
	\[
	\|u_0\|^2_{L^2(\mathbb T^d)}
	\le
	C\int_0^T
	\|\operatorname{Re}e^{it(\Delta-V)}u_0\|^2_{L^2(E)}\,\d t,
	\qquad
	\forall u_0\in
	\mathcal K_{\operatorname{Im}}^{\perp_{\mathbb R}}.
	\]
\end{proposition}

We first prove \lemref{lem-ucp}.

\begin{proof}[Proof of \lemref{lem-ucp}]
	We first prove the easy inclusion. If
	\(\phi\in\mathcal K_{\operatorname{Im}}\), then \(H\phi=0\) and
	\(\operatorname{Re}\phi=0\). Hence, by the spectral calculus,
	\(e^{-itH}\phi=\phi\) for all \(t\in\R\), and therefore
	\[
		\operatorname{Re}e^{-itH}\phi=0
		\quad\text{in }L^2(E),\qquad t\in(0,T).
	\]
	Thus
	\begin{equation}\label{eq:ucp-Kim-subset-NT}
		\mathcal K_{\operatorname{Im}}\subset\mathcal N_T .
	\end{equation}

	We prove the converse in two steps.

\emph{Step 1. Invariance under the real generator and finite dimensionality.}

Set \(A=-iH\). We first prove that 
\begin{equation}\label{eq:ucp-A-invariance}
\mathcal N_T\subset D(A),
\qquad	A\mathcal N_T\subset\mathcal N_T .
\end{equation}

Let \(u_0\in\mathcal N_T\) and
\begin{equation}\label{3.29-7-15-w}
	u_\varepsilon
	=
	\frac{e^{-i\varepsilon H}u_0-u_0}{\varepsilon},
	\qquad
	0<\varepsilon<T/2 .
\end{equation}
Then, for a.e. \(s\in(0,T/2)\),
\begin{equation}\label{3.30-7-15-wg}
\operatorname{Re}(e^{-isH}u_\varepsilon)|_E
=
\frac{
\operatorname{Re}(e^{-i(s+\varepsilon)H}u_0)|_E
-
\operatorname{Re}(e^{-isH}u_0)|_E
}{\varepsilon}
=0\qquad\mbox{in}\; L^2(E;\mathbb R).
\end{equation}

Moreover, by  the spectral theorem,  \eqref{3.29-7-15-w} and
\(
|e^{-i\varepsilon\lambda}-1|\leq \varepsilon|\lambda|,
\)
we obtain
\begin{equation}\label{eq:ucp-ueps-Hminus2}
\|u_\varepsilon\|_{\mathcal H^{-2}}
\leq C\|u_0\|_{L^2(\mathbb T^d)} .
\end{equation}
Then, by  \eqref{3.30-7-15-wg} and \eqref{eq:ucp-ueps-Hminus2}, we apply
Proposition~\ref{prop-compact-remainder}, with \(T/2\), to obtain
\[
\|u_\varepsilon\|_{L^2(\mathbb T^d)}
\leq C\|u_0\|_{L^2(\mathbb T^d)} .
\]
Therefore, for some \(\varepsilon_j\to0\),
\begin{equation}\label{eq:ucp-weak-convergence}
u_{\varepsilon_j}\rightharpoonup w
\quad\text{in the real Hilbert space }L^2(\mathbb T^d;\mathbb C).
\end{equation}

Let \(H\varphi_n=\lambda_n\varphi_n\). Then
\begin{equation}\label{3.33-7-15-wwgg}
	(u_\varepsilon,\varphi_n)
	=
	\frac{e^{-i\varepsilon\lambda_n}-1}{\varepsilon}
	(u_0,\varphi_n),
	\qquad n\geq1 .
\end{equation}
Testing \eqref{eq:ucp-weak-convergence} with
\(\varphi_n\) and \(i\varphi_n\), and using \eqref{3.33-7-15-wwgg}, we get
\begin{equation}\label{eq:ucp-spectral-limit}
	(w,\varphi_n)
	=
	-i\lambda_n(u_0,\varphi_n),
	\qquad n\geq1 .
\end{equation}
By \eqref{eq:ucp-spectral-limit},
\[
	\lambda_n^2|(u_0,\varphi_n)|^2
	=
	|(w,\varphi_n)|^2,
	\qquad n\geq1 .
\]
Since \(w\in L^2(\mathbb T^d;\mathbb C)\), Parseval's identity yields
\[
	\sum_n\lambda_n^2|(u_0,\varphi_n)|^2
	=
	\sum_n|(w,\varphi_n)|^2
	=
	\|w\|_{L^2(\mathbb T^d)}^2
	<\infty .
\]
Hence \(u_0\in D(H)\). Moreover, by the spectral representation of \(H\) and
\eqref{eq:ucp-spectral-limit},
\begin{equation}\label{eq:ucp-weak-limit-Au0}
	w=Au_0=-iHu_0
	\quad\text{in }L^2(\mathbb T^d;\mathbb C).
\end{equation}

Let \(I\Subset(0,T)\). For large \(j\),
\(s,s+\varepsilon_j\in(0,T)\) for \(s\in I\), hence
\[
\operatorname{Re}(e^{-isH}u_{\varepsilon_j})|_E=0
\quad\text{in }L^2(I\times E;\mathbb R).
\]
Passing to the weak limit by the boundedness of
\[
v\mapsto e^{-isH}v:
L^2(\mathbb T^d;\mathbb C)\to L^2(I\times E;\mathbb C),
\]
and using \eqref{eq:ucp-weak-limit-Au0}, we obtain
\[
\operatorname{Re}(e^{-isH}Au_0)|_E=0
\quad\text{in }L^2(I\times E;\mathbb R).
\]
Since \(I\Subset(0,T)\) was arbitrary, we obtain \(Au_0\in\mathcal N_T\), and thus
\eqref{eq:ucp-A-invariance}.

Finally, \(\mathcal N_T\) is closed since
\[
v\mapsto\operatorname{Re}(e^{-itH}v)|_E
\]
is bounded from \(L^2(\mathbb T^d;\mathbb C)\) to
\(L^2((0,T)\times E;\mathbb R)\).
Moreover, by \eqref{eq:ucp-A-invariance} and Proposition~\ref{prop-compact-remainder}, we see
\[
\|Hu_0\|_{L^2(\mathbb T^d)}
\leq C\|Au_0\|_{\mathcal H^{-2}}
\leq C\|u_0\|_{L^2(\mathbb T^d)},\;\;\mbox{for each}\;u_0\in\mathcal N_T.
\]
Hence the unit ball of \(\mathcal N_T\) is bounded in the graph norm of
\(H\). Since \(H\) has compact resolvent, this unit ball is relatively
compact in \(L^2(\mathbb T^d;\mathbb C)\).  
Therefore, by Riesz's lemma,
\begin{equation}\label{eq:ucp-NT-finite-dimensional}
	\dim_{\mathbb R}\mathcal N_T<\infty .
\end{equation}

\emph{Step 2. Identification of \(\mathcal N_T\).}

The aim of this step is to prove
\(\mathcal N_T\subset\mathcal K_{\operatorname{Im}}\).
To this end, we first prove that \(A_T:=A|_{\mathcal N_T}=0\), which yields
\(\mathcal N_T\subset\ker H\). The desired inclusion then follows from the
definition of \(\mathcal N_T\) and elliptic unique continuation.

By \eqref{eq:ucp-A-invariance} and
\eqref{eq:ucp-NT-finite-dimensional}, \(\mathcal N_T\) is a finite-dimensional
real Hilbert space invariant under \(A=-iH\). For \(u,v\in\mathcal N_T\), the self-adjointness of \(H\) gives
\[
\langle Au,v\rangle_{\mathbb R}
=
-\langle u,Av\rangle_{\mathbb R}.
\]
Thus \(A_T\) is skew-symmetric on \(\mathcal N_T\).

We first show that \(A_T\) has no nonzero rotation block. Suppose otherwise.
Then there exist linearly independent
\(
	p,q\in\mathcal N_T\subset L^2(\mathbb T^d;\mathbb C)
\) and \(\lambda>0\) such that
\begin{equation}\label{eq:ucp-rotation-block}
	Ap=\lambda q,
	\qquad
	Aq=-\lambda p
	\quad\text{in }L^2(\mathbb T^d;\mathbb C).
\end{equation}
Hence
\begin{align}\label{equ-83-1}
	e^{tA}p=
\cos(\lambda t)p+\sin(\lambda t)q,
\qquad t\in\mathbb R .	
\end{align} 
Since \(p\in\mathcal N_T\), the definition of \(\mathcal N_T\) and the
identity \eqref{equ-83-1} give
\[
	\cos(\lambda t)\operatorname{Re}p|_E
	+
	\sin(\lambda t)\operatorname{Re}q|_E
	=
	0
	\quad\text{in }L^2(E;\mathbb R),
	\qquad t\in(0,T).
\]
Taking the \(L^2(0,T)\)-pairings with
\(\cos(\lambda t)\) and \(\sin(\lambda t)\), and using the invertibility of
the corresponding Gram matrix, we obtain
\begin{equation}\label{eq:ucp-Re-pq-zero-on-E}
	\operatorname{Re}p|_E=0,
	\qquad
	\operatorname{Re}q|_E=0
	\quad\text{in }L^2(E;\mathbb R).
\end{equation}
Therefore,
\begin{equation}\label{eq:ucp-pq-imag-on-E}
	p|_E,q|_E\in L^2(E;i\mathbb R).
\end{equation}

By \eqref{eq:ucp-rotation-block} and \(A=-iH\),
\begin{equation}\label{eq:ucp-Hpq-relations}
	Hp=i\lambda q,
	\qquad
	Hq=-i\lambda p
	\quad\text{in }L^2(\mathbb T^d;\mathbb C).
\end{equation}
By \eqref{eq:ucp-pq-imag-on-E} and
\eqref{eq:ucp-Hpq-relations},
\begin{equation}\label{eq:ucp-Hpq-real-on-E}
	Hp|_E,\ Hq|_E\in L^2(E;\mathbb R).
\end{equation}
Since \(H\) has real coefficients, \eqref{eq:ucp-Re-pq-zero-on-E} implies
\[
	\operatorname{Re}(Hp)|_E=0,
	\qquad
	\operatorname{Re}(Hq)|_E=0
	\quad\text{in }L^2(E;\mathbb R).
\]
Combining this with \eqref{eq:ucp-Hpq-real-on-E}, we obtain
\begin{equation}\label{eq:ucp-Hpq-zero-on-E}
	Hp|_E=0,
	\qquad
	Hq|_E=0
	\quad\text{in }L^2(E;\mathbb R).
\end{equation}
By \eqref{eq:ucp-Hpq-relations} and
\eqref{eq:ucp-Hpq-zero-on-E},
\begin{equation}\label{eq:ucp-pq-zero-on-E}
	p|_E=0,
	\qquad
	q|_E=0
	\quad\text{in }L^2(E;\mathbb C).
\end{equation}

Set
\[
	w_+=p+iq,\qquad w_-=p-iq .
\]
Then
\[
	(H-\lambda)w_+=0,
	\qquad
	(H+\lambda)w_-=0
	\quad\text{in }L^2(\mathbb T^d;\mathbb C).
\]
Moreover, by \eqref{eq:ucp-pq-zero-on-E},
\[
	w_+|_E=w_-|_E=0
	\quad\text{in }L^2(E;\mathbb C).
\]
The standard elliptic unique continuation property \cite{ScheSimon80} yields
\[
	w_+=w_-=0
	\quad\text{in }L^2(\mathbb T^d;\mathbb C),
\]
which contradicts the choice of a nonzero rotation block.

Therefore \(A_T\) has no nonzero rotation block. By the finite-dimensional
skew-adjoint normal form, we have $	A_T=0$. This, along with definitions of $A_T$ and $\mathcal N_T$, yields
\begin{equation}\label{eq:ucp-NT-subset-K}
	\mathcal N_T\subset\ker H=\mathcal K .
\end{equation}

Let \(\phi\in\mathcal N_T\). By \eqref{eq:ucp-NT-subset-K} and the spectral
theorem, we obtain
\begin{equation}\label{3.44-7-14-w}
	H\phi=0,
	\qquad
	e^{-itH}\phi=\phi,
	\quad t\in\mathbb R .
\end{equation}
The second identity in \eqref{3.44-7-14-w} and the definition of
\(\mathcal N_T\) yield
\begin{equation}\label{eq:ucp-Rephi-zero-on-E}
	\operatorname{Re}\phi|_E=0
	\quad\text{in }L^2(E;\mathbb R).
\end{equation}
Since \(H\) has real coefficients, the first identity in
\eqref{3.44-7-14-w} implies
\[
	H(\operatorname{Re}\phi)=0 .
\]
By the elliptic unique continuation again and
\eqref{eq:ucp-Rephi-zero-on-E}, we obtain
\begin{equation*} 
	\operatorname{Re}\phi=0
	\quad\text{in }L^2(\mathbb T^d;\mathbb R).
\end{equation*}
Hence
\[
\phi=i\operatorname{Im}\phi
\quad\text{a.e. on }\mathbb T^d.
\]
Since \(\phi\in\mathcal K\), the definition
\eqref{equ-ker-ReIm} directly gives \(\phi\in\mathcal K_{\operatorname{Im}}.\)
Therefore
\begin{equation}\label{eq:ucp-NT-subset-Kim}
	\mathcal N_T\subset\mathcal K_{\operatorname{Im}}.
\end{equation}

Combining \eqref{eq:ucp-Kim-subset-NT} and
\eqref{eq:ucp-NT-subset-Kim}, we conclude that
$	\mathcal N_T=\mathcal K_{\operatorname{Im}}$.
This completes the proof.	
\end{proof}

Now we turn to prove \propref{prop-ob-M}. We note that
 \begin{equation}\label{equ-316-2}
	L^2(\mathbb T^d;\mathbb C)
	=
	\mathcal K_{\operatorname{Im}}^{\perp_{\mathbb R}}
	\oplus^{\perp_{\mathbb R}}
	\mathcal K_{\operatorname{Im}} .
\end{equation}

\begin{proof}[Proof of \propref{prop-ob-M}]
	We argue by contradiction. Suppose that the estimate is false. Then there
	exists a sequence
	\(\{u_n\}_{n\geq1}\subset \mathcal K_{\operatorname{Im}}^{\perp_{\mathbb R}}\)
	such that
	\begin{equation*}
		\|u_n\|_{L^2(\mathbb T^d)}=1,
		\qquad n\geq1,
	\end{equation*}
	and
	\begin{equation}\label{3.29-w-7-9}
		\int_0^T
		\|\operatorname{Re}e^{it(\Delta-V)}u_n\|^2_{L^2(E)}
		\,\d t\longrightarrow0,
		\qquad\text{as }n\to\infty .
	\end{equation}

	Since \(\{u_n\}_{n\geq1}\) is bounded in
	\(L^2(\mathbb T^d;\mathbb C)\), after passing to a subsequence, we may assume
	that
	\[
		u_n\rightharpoonup u
		\quad\text{weakly in }L^2(\mathbb T^d;\mathbb C),
		\qquad\text{as }n\to\infty .
	\]
	By the weak closedness of the real subspace
\(\mathcal K_{\operatorname{Im}}^{\perp_{\mathbb R}}\) of the real Hilbert
space \(L^2(\mathbb T^d;\mathbb C)\), we have
\begin{equation}\label{3.53-7-15-ww}
u\in\mathcal K_{\operatorname{Im}}^{\perp_{\mathbb R}}
\end{equation}
	Moreover, the compact embedding
	\(L^2(\mathbb T^d;\mathbb C)\hookrightarrow\mathcal H^{-2}\) implies, after
	extracting a further subsequence, that
	\begin{equation}\label{3.54-7-15-ww}
		u_n\to u
		\quad\text{strongly in }\mathcal H^{-2},
		\qquad\text{as }n\to\infty .
	\end{equation}

	We next prove that \(u\in\mathcal N_T\). Indeed, the observation map
	\[
	v\mapsto
	\operatorname{Re}(e^{it(\Delta-V)}v)|_E
	\]
	is bounded and real-linear from
	\(L^2(\mathbb T^d;\mathbb C)\) into
	\(L^2((0,T)\times E;\mathbb R)\). Therefore,
	\[
		\operatorname{Re}e^{it(\Delta-V)}u_n
		\rightharpoonup
		\operatorname{Re}e^{it(\Delta-V)}u
		\quad\text{weakly in }L^2((0,T)\times E;\mathbb R),
		\qquad\text{as }n\to\infty .
	\]
	On the other hand, by \eqref{3.29-w-7-9},
	\[
		\operatorname{Re}e^{it(\Delta-V)}u_n
		\to0
		\quad\text{strongly in }L^2((0,T)\times E;\mathbb R),
		\qquad\text{as }n\to\infty .
	\]
	Thus \( \operatorname{Re}e^{it(\Delta-V)}u=0\) in
	\(L^2((0,T)\times E;\mathbb R)\). By the definition of \(\mathcal N_T\) in \eqref{equ-NT-definition},
	we have \(u\in\mathcal N_T\). Hence, by \lemref{lem-ucp},
	\(u\in\mathcal K_{\operatorname{Im}}\). This, together with \eqref{3.53-7-15-ww}, implies
	\begin{equation}\label{3.55-7-15-wwgg}
		u\in\mathcal K_{\operatorname{Im}}
		\cap\mathcal K_{\operatorname{Im}}^{\perp_{\mathbb R}} .
	\end{equation}
	By \eqref{3.55-7-15-wwgg} and \eqref{equ-316-2}, we obtain \(u=0\).
This, along with \eqref{3.54-7-15-ww}, yields	
	\begin{equation}\label{3.50-7-12-w}
		u_n\to0
		\quad\text{strongly in }\mathcal H^{-2},
		\qquad\text{as }n\to\infty .
	\end{equation}

	Finally, applying the compact-remainder estimate of
	\propref{prop-compact-remainder}, we have
	\[
		1=\|u_n\|^2_{L^2(\mathbb T^d)}
		\leq
		C\int_0^T
		\|\operatorname{Re}e^{it(\Delta-V)}u_n\|^2_{L^2(E)}
		\,\d t
		+
		C\|u_n\|^2_{\mathcal H^{-2}}.
	\]
	The first term on the right-hand side tends to zero by 
	\eqref{3.29-w-7-9}, and the second one tends to zero by \eqref{3.50-7-12-w}.
	 This
	contradiction proves the desired estimate.
\end{proof}

\subsection[Proof of the open-set real-part estimate]
{Proof of \texorpdfstring{\thmref{thm-obRe-V}}
	{the open-set real-part estimate}}\label{subsec:proof-open-real-part-estimate}

	Let \(u_0\in L^2(\mathbb T^d;\mathbb C)\). By the real orthogonal
	decomposition \eqref{equ-316-2}, we write
	\begin{equation}\label{eq:thm-obRe-decomposition}
		u_0=u_0^\perp+u_0^{\operatorname{Im}},
	\end{equation}
	where
	\begin{equation}\label{3.54-7-15-wgs}
		u_0^\perp\in\mathcal K_{\operatorname{Im}}^{\perp_{\mathbb R}},
		\qquad
		u_0^{\operatorname{Im}}
		=
		P_{\mathcal K_{\operatorname{Im}}}u_0
		\in\mathcal K_{\operatorname{Im}} .
	\end{equation}

	By the definition of \(\mathcal K_{\operatorname{Im}}\), we have
	\[
		u_0^{\operatorname{Im}}\in\ker(-\Delta+V),
		\qquad
		\operatorname{Re}u_0^{\operatorname{Im}}=0 .
	\]
	By the spectral theorem,
	\[
		e^{it(\Delta-V)}u_0^{\operatorname{Im}}
		=
		u_0^{\operatorname{Im}} .
	\]
This, along with \eqref{eq:thm-obRe-decomposition}, yields
	\begin{equation}\label{eq:thm-obRe-observation-perp}
\operatorname{Re}e^{it(\Delta-V)}u_0
=
\operatorname{Re}e^{it(\Delta-V)}u_0^\perp
\quad\text{in }L^2((0,T)\times E;\mathbb R).
	\end{equation}

On the other hand, by \propref{prop-ob-M} and the first identity in \eqref{3.54-7-15-wgs}, we obtain
	\begin{equation}\label{eq:thm-obRe-perp-estimate}
		\|u_0^\perp\|^2_{L^2(\mathbb T^d)}
		\leq
		C\int_0^T
		\|\operatorname{Re}e^{it(\Delta-V)}u_0^\perp\|^2_{L^2(E)}
		\,\d t .
	\end{equation}
	By \eqref{eq:thm-obRe-observation-perp} and
	\eqref{eq:thm-obRe-perp-estimate}, we obtain
	\begin{equation}\label{3.55-7-15-w}
		\|u_0^\perp\|^2_{L^2(\mathbb T^d)}
		\leq
		C\int_0^T
		\|\operatorname{Re}e^{it(\Delta-V)}u_0\|^2_{L^2(E)}
		\,\d t .
	\end{equation}
Since
		\[
		\|u_0\|^2_{L^2(\mathbb T^d)}
		=
		\|u_0^\perp\|^2_{L^2(\mathbb T^d)}
		+
		\|P_{\mathcal K_{\operatorname{Im}}}u_0\|^2_{L^2(\mathbb T^d)}, 
	\]
it follows from	 \eqref{3.55-7-15-w}  that
	\[
		\|u_0\|^2_{L^2(\mathbb T^d)}
		\leq
		C\int_0^T
		\|\operatorname{Re}e^{it(\Delta-V)}u_0\|^2_{L^2(E)}
		\,\d t
		+
		\|P_{\mathcal K_{\operatorname{Im}}}u_0\|^2_{L^2(\mathbb T^d)} .
	\]
	This proves the estimate in \thmref{thm-obRe-V}.

	It remains to prove the optimality of the correction term. Let \(Y\) be a
	proper real subspace of \(\mathcal K_{\operatorname{Im}}\). Since
	\(
		Y\subsetneq\mathcal K_{\operatorname{Im}},
	\)
	there exists
	\[
		z\in\mathcal K_{\operatorname{Im}}\cap Y^{\perp_{\mathbb R}},
		\qquad z\neq0 .
	\]
	By the definition of \(\mathcal K_{\operatorname{Im}}\) in \eqref{equ-ker-ReIm} and the spectral
	theorem,
	\[
		\operatorname{Re}e^{it(\Delta-V)}z=0\quad\text{on }(0,T)\times\mathbb T^d .
			\]
Moreover, since \(z\in Y^{\perp_{\mathbb R}}\), the orthogonal projection
of \(z\) onto \(Y\) is zero.	
 Therefore, replacing the correction term by the
	projection onto \(Y\) would give a zero right-hand side for \(u_0=z\),
	whereas
	\(
		\|z\|^2_{L^2(\mathbb T^d)}>0 .
	\)
	This contradiction proves that the correction term cannot be replaced by
	the projection onto any proper real subspace of
	\(\mathcal K_{\operatorname{Im}}\).

\section{Real-part observability from admissible measurable product sets}\label{sec:measurable-real-part}
We use the notation introduced in \secref{sec:introduction}. In particular,
the spaces \(\mathcal K_c\), \((\mathcal K_c)_{\operatorname{Im}}\), the
projection \(P_{(\mathcal K_c)_{\operatorname{Im}}}\), and the time slab \(Q\)
are given by
\eqref{1.8-7-14-w}, \eqref{1.9-7-14-w},
\eqref{eq:P-KcIm-definition}, and \eqref{eq:time-slab-Q}, respectively.
The admissible measurable product structure is described by
\eqref{1.8-7-27-w}--\eqref{1.10-7-27-w}, and the associated block variables
are defined by \eqref{eq:block-variable-definition}.

\begin{theorem}[Real-part observability from admissible measurable product sets]
	\label{thm:full}
	Let \(G\subset Q\) be an admissible measurable product set, and let
	\(c\in\mathbb R\). Then there exists a constant \(C>0\), depending only on
	\(d\), \(c\), and \(G\), such that
	\begin{equation}\label{equ-measurable-realpart}
		\|\phi\|_{L^2_x}^2
		\leq
		C\|\operatorname{Re}e^{it(\Delta+c)}\phi\|_{L^2(G)}^2
		+
		\|P_{(\mathcal K_c)_{\operatorname{Im}}}\phi\|_{L^2_x}^2,
		\qquad
		\forall \phi\in L^2(\mathbb T^d;\mathbb C).
	\end{equation}
\end{theorem}

\begin{corollary}[Real-part observability without correction]
	\label{cor:real-part-iff}
	Let \(G\subset Q\) be an admissible measurable product set, and let
	\(c\in\mathbb R\). Then
	\[
	\|\phi\|_{L^2_x}^2
	\leq
		C\|\operatorname{Re}e^{it(\Delta+c)}\phi\|_{L^2(G)}^2,
		\qquad
		\forall \phi\in L^2(\mathbb T^d;\mathbb C),
	\]
	holds for some constant \(C>0\) if and only if
	\[
		c\notin\{|n|^2:n\in\mathbb Z^d\}.
	\]
\end{corollary}

\subsection{Double-spectrum reduction and external observability inputs}
\label{subsec:double-spectrum-reduction}

The purpose of this subsection is to reduce the proof of
\thmref{thm:full} to a full-frequency observability estimate on the
double-spectrum space associated with the real-part observation. Since taking
the real part introduces the two time-frequency branches associated with the
shifted dispersion relation, the real-part observation is not a single
Schr\"odinger spectral branch, but involves the double spectrum
\(\Sigma_c\) given by \eqref{1.18-7-28-w}. We first introduce the corresponding
double-spectrum space and then establish the full-frequency observability
estimate on this space by combining a high-frequency estimate with a
low-frequency absorption argument.

Although the time interval has length \(2\pi\), we do not identify the time
variable periodically in the shifted case. Indeed, \(e^{it\Delta}\) is
\(2\pi\)-periodic in time, whereas
\(e^{it(\Delta+c)}=e^{ict}e^{it\Delta}\) is \(2\pi\)-periodic only when
\(c\in\mathbb Z\). Throughout this section, the time variable is therefore
treated as a variable on the interval \((0,2\pi)\), and the spacetime domain is
the slab \(Q\) defined by \eqref{eq:time-slab-Q}.

We use the block notation introduced in \secref{sec:introduction}, where
\(z_{\alpha_j}\) is defined by
\eqref{eq:block-variable-definition}. For each \(j\), we also use the notation
\begin{equation}\label{4.2-4-30-w}
	z_{\widehat{\alpha_j}}:=(z_\ell)_{\ell\notin\alpha_j},
	\qquad
	z=(z_{\alpha_j},z_{\widehat{\alpha_j}}).
\end{equation}
The two-dimensional block estimates are the main input, while the
one-dimensional block estimates follow from the two-dimensional ones by
adjoining an auxiliary coordinate.

For \(\phi\in L^2(\mathbb T^d;\mathbb C)\), we have
\begin{equation}\label{eq:shifted-free-expansion}
	e^{it(\Delta+c)}\phi(x)
	=
	\sum_{n\in\mathbb Z^d}
	\widehat\phi(n)e^{i(n\cdot x-\mu_n t)},
	\qquad
	\mu_n=|n|^2-c ,
\end{equation}
where the series converges in \(L^2(Q)\). It follows from
\eqref{eq:shifted-free-expansion} that the real-part observation belongs to
the double-spectrum space
\begin{equation}\label{eq:double-spectrum-space}
	\mathcal X
	=
	\overline{\operatorname{span}}^{\,L^2(Q)}
	\left\{
	e^{i(n\cdot x-\mu_n t)},
	e^{i(n\cdot x+\mu_n t)}
	:\ n\in\mathbb Z^d
	\right\}.
\end{equation}
When \(\mu_n=0\), the two generators coincide and are counted only once.

The following theorem gives the main observability input for the proof of
\thmref{thm:full}.

\begin{theorem}[Full-frequency double-spectrum inequality]
	\label{thm:double-spectrum-obs}
	Let  \(\mathcal X\) be defined by \eqref{eq:double-spectrum-space}.
There exists a constant \(C_G>0\), depending only on \(d\), \(c\), and
	\(G\), such that
	\begin{equation}\label{eq:full-frequency-X}
		\|F\|_{L^2(Q)}^2
		\leq
		C_G\|F\|_{L^2(G)}^2,
		\qquad
		\forall F\in\mathcal X .
	\end{equation}

	In particular, for every
	\(\phi\in L^2(\mathbb T^d;\mathbb C)\),
	\begin{equation}\label{eq:real-part-QG}
		\|\operatorname{Re} e^{it(\Delta+c)}\phi\|_{L^2(Q)}^2
		\leq
		C_G
		\|\operatorname{Re} e^{it(\Delta+c)}\phi\|_{L^2(G)}^2 .
	\end{equation}
\end{theorem}

To prove \thmref{thm:double-spectrum-obs}, we use the following
complex-valued estimates of Burq--Zhu. These results provide observability from admissible measurable product
sets, together with auxiliary estimates for rough product observation sets.
In the proof of \thmref{thm:double-spectrum-obs}, they will be combined with
the double-spectrum analysis generated by the real-part observation and with
the low-frequency absorption argument developed below. We record only the
forms needed in the sequel.

\begin{proposition}[External inputs from complex-valued observability]
	\label{prop:external-inputs}
	Let \(G\subset Q\) be an admissible measurable product set, and let
	\(\chi=\mathbf 1_G\). The following estimates hold.

	\begin{itemize}
		\item [(i)] There exists a constant \(C_{\rm obs}>0\), depending only on
		\(d\) and \(G\), such that
		\[
		\|e^{it\Delta}f\|_{L^2(Q)}
		\leq
		C_{\rm obs}
		\|e^{it\Delta}f\|_{L^2(G)},
		\qquad
		\forall f\in L^2(\mathbb T^d;\mathbb C).
		\]

		\item [(ii)] There exists a constant \(C_Z>0\), depending only on \(d\),
such that, for every block \(\alpha_j\) in \eqref{1.9-7-27-w}, one has
\[
\|e^{it\Delta}f\|_{L^4_{z_{\alpha_j}}
L^2_{z_{\widehat{\alpha_j}}}}
\leq
C_Z\|f\|_{L^2(\mathbb T^d)},
\qquad
\forall f\in L^2(\mathbb T^d;\mathbb C),
\]
where \(z_{\alpha_j}\) and \(z_{\widehat{\alpha_j}}\) are defined in
\eqref{eq:block-variable-definition} and \eqref{4.2-4-30-w}, respectively.
		
\item [(iii)] For every \(\varepsilon>0\), there exists a real-valued
		trigonometric polynomial
		\[
		\zeta(t,x)
		=
		\sum_{q\in\mathcal Q_R}
		\zeta_q e^{i(q_0t+q'\cdot x)},
		\qquad
		\mathcal Q_R\subset\mathbb Z^{1+d}\text{ finite},
		\]
		such that
		\begin{equation}\label{eq:forward-multiplier-estimate}
		\|(\chi-\zeta)e^{it\Delta}f\|_{L^2(Q)}
		\leq
		\varepsilon
		\|e^{it\Delta}f\|_{L^2(Q)},
		\qquad
		\forall f\in L^2(\mathbb T^d;\mathbb C),
		\end{equation}
		and
		\begin{equation}\label{eq:backward-multiplier-estimate}
		\|(\chi-\zeta)e^{-it\Delta}f\|_{L^2(Q)}
		\leq
		\varepsilon
		\|e^{-it\Delta}f\|_{L^2(Q)},
		\qquad
		\forall f\in L^2(\mathbb T^d;\mathbb C).
		\end{equation}
		
		Moreover, for every \(c\in\mathbb R\) and every
		\(f,g\in L^2(\mathbb T^d;\mathbb C)\), the same polynomial \(\zeta\)
		satisfies
		\begin{equation}\label{eq:shifted-double-branch-estimate}
		\begin{aligned}
		&\|(\chi-\zeta)
		(e^{ict}e^{it\Delta}f+e^{-ict}e^{-it\Delta}g)\|_{L^2(Q)}
		\\
		&\qquad\leq
		\varepsilon
		\Big(
		\|e^{ict}e^{it\Delta}f\|_{L^2(Q)}
		+
		\|e^{-ict}e^{-it\Delta}g\|_{L^2(Q)}
		\Big).
		\end{aligned}
		\end{equation}
	\end{itemize}
\end{proposition}
\begin{proof}
		(i) The estimate follows from the complex-valued observability theory of
Burq--Zhu; see Theorems~1.2 and~1.10 in
\cite{BurqZhu-spacetime-meas}.
	
(ii) We recall that \(Q_\ell\) and \(Q_{\alpha_j}\) are defined in
\secref{sec:introduction}, and that \(z_{\widehat{\alpha_j}}\) is defined by
\eqref{4.2-4-30-w}. If \(\#\alpha_j=2\), the estimate is the
Zygmund-type mixed-norm estimate of Burq--Zhu; see Lemma~2.5 of
\cite{BurqZhu-spacetime-meas}.

If \(\#\alpha_j=1\), choose
\(\ell\in\{0,\ldots,d\}\setminus\alpha_j\), set
\[
	\beta:=\alpha_j\cup\{\ell\},
	\qquad
	z_\beta:=(z_r)_{r\in\beta},
\]
and write \(w=z_\ell\). For \(u=e^{it\Delta}f\), the Cauchy--Schwarz
inequality in the \(w\)-variable gives
\[
\begin{aligned}
\|u\|_{L^4_{z_{\alpha_j}}
L^2_{z_{\widehat{\alpha_j}}}}^4
&=
\int_{Q_{\alpha_j}}
\left(
\int_{Q_\ell}
\int_{\prod_{r\notin\beta}Q_r}|u|^2
\right)^2
\,\d z_{\alpha_j}
\\
&\leq
|Q_\ell|
\int_{Q_{\alpha_j}}\int_{Q_\ell}
\left(
\int_{\prod_{r\notin\beta}Q_r}|u|^2
\right)^2
\,\d z_\ell\,\d z_{\alpha_j}
\\
&=
|Q_\ell|\,
\|u\|_{L^4_{z_\beta}L^2_{(z_r)_{r\notin\beta}}}^4 .
\end{aligned}
\]
The two-dimensional estimate applied to the coordinate pair \(z_\beta\) then
gives the desired one-dimensional estimate, after enlarging the constant.

(iii) We prove \eqref{eq:forward-multiplier-estimate},
\eqref{eq:backward-multiplier-estimate}, and
\eqref{eq:shifted-double-branch-estimate} separately.

\medskip
\noindent
\emph{Step 1. We prove the forward estimate
\eqref{eq:forward-multiplier-estimate}.}

By Theorem~1.10 of Burq--Zhu~\cite{BurqZhu-spacetime-meas}, together with
the estimate in (ii), the function \(\chi=\mathbf 1_G\) belongs to the
closure, in the multiplier norm on the forward Schr\"odinger solution space,
of trigonometric polynomials. More precisely, let \(\mathcal P_{\rm trig}\) denote the space of
trigonometric polynomials on \(Q\). For a measurable function \(m\) on
\(Q\), define its multiplier norm on the forward Schr\"odinger solution
space by
\[
\|m\|_{\mathfrak M_+}
:=
\sup_{0\neq f\in L^2(\mathbb T^d;\mathbb C)}
\frac{
	\|m e^{it\Delta}f\|_{L^2(Q)}
}{
	\|e^{it\Delta}f\|_{L^2(Q)}
}.
\]
Then
\[
	\inf_{\zeta\in\mathcal P_{\rm trig}}
	\|\chi-\zeta\|_{\mathfrak M_+}
	=
	0 .
\]
Hence, for every \(\varepsilon>0\), there exists a trigonometric polynomial
\(\zeta_0\) such that
\begin{equation}\label{eq:forward-complex-polynomial}
	\|(\chi-\zeta_0)e^{it\Delta}f\|_{L^2(Q)}
	\le
	\varepsilon
	\|e^{it\Delta}f\|_{L^2(Q)},
	\qquad
	\forall f\in L^2(\mathbb T^d;\mathbb C).
\end{equation}
Since \(\chi\) is real-valued, replacing \(\zeta_0\) by
\(\zeta=\operatorname{Re}\zeta_0\) does not increase the multiplier error.
Indeed,
\[
|\chi-\operatorname{Re}\zeta_0|
=
|\operatorname{Re}(\chi-\zeta_0)|
\le
|\chi-\zeta_0|
\quad\text{a.e. on }Q .
\]
Thus \eqref{eq:forward-complex-polynomial} implies
\eqref{eq:forward-multiplier-estimate} with a real-valued trigonometric
polynomial \(\zeta\).

\medskip
\noindent
\emph{Step 2. We prove 
\eqref{eq:backward-multiplier-estimate}.}

We use the same real-valued trigonometric polynomial \(\zeta\) obtained in
Step~1. Since
\[
	e^{-it\Delta}f
	=
	\overline{e^{it\Delta}\overline f},
\]
and \(\chi-\zeta\) is real-valued, we have
\[
\begin{aligned}
	\|(\chi-\zeta)e^{-it\Delta}f\|_{L^2(Q)}
	=
	\|(\chi-\zeta)e^{it\Delta}\overline f\|_{L^2(Q)}
	\leq
	\varepsilon
	\|e^{-it\Delta}f\|_{L^2(Q)},\;\;f\in L^2(\mathbb T^d;\mathbb C).
\end{aligned}
\]
 This proves
\eqref{eq:backward-multiplier-estimate}.

\medskip
\noindent
\emph{Step 3. We prove \eqref{eq:shifted-double-branch-estimate}.}

Fix \(c\in\mathbb R\) and
\(f,g\in L^2(\mathbb T^d;\mathbb C)\).
 Since \(e^{\pm ict}\) are
unimodular and commute with multiplication by \(\chi-\zeta\), the triangle
inequality, \eqref{eq:forward-multiplier-estimate}, and
\eqref{eq:backward-multiplier-estimate} give
\[
\begin{aligned}
&\|(\chi-\zeta)
(e^{ict}e^{it\Delta}f+e^{-ict}e^{-it\Delta}g)\|_{L^2(Q)}
\\
&\qquad\le
\|(\chi-\zeta)e^{ict}e^{it\Delta}f\|_{L^2(Q)}
+
\|(\chi-\zeta)e^{-ict}e^{-it\Delta}g\|_{L^2(Q)}
\\
&\qquad\le
\varepsilon
\Big(
\|e^{ict}e^{it\Delta}f\|_{L^2(Q)}
+
\|e^{-ict}e^{-it\Delta}g\|_{L^2(Q)}
\Big).
\end{aligned}
\]
This proves \eqref{eq:shifted-double-branch-estimate}, and completes the proof.	
\end{proof}

\subsection{High-frequency observability}
\label{subsec:measurable-high-frequency}
We first isolate the high-frequency part of the double-spectrum space. For
\(N>0\), let \(\mathcal X_{>N}\) be the \(L^2(Q)\)-closure of all functions
of the form
\[
	F=u_+ + u_-,
\]
where
\[
	u_+(t,x)=e^{ict}e^{it\Delta}\phi(x),
	\qquad
	u_-(t,x)=e^{-ict}e^{-it\Delta}\psi(x),
\]
with
\[
	\phi,\psi\in L^2(\mathbb T^d;\mathbb C),
	\qquad
	\operatorname{supp}\widehat\phi,\operatorname{supp}\widehat\psi
	\subset
	\{n\in\mathbb Z^d:\ |n|>N\}.
\]
Equivalently, \(\mathcal X_{>N}\) is the \(L^2(Q)\)-closure of finite sums
\[
F(t,x)
=
\sum_{|n|>N}a_n^+e^{i(n\cdot x-\mu_n t)}
+
\sum_{|n|>N}a_n^-e^{i(n\cdot x+\mu_n t)},
\]
where only finitely many coefficients \(a_n^\pm\) are nonzero, and
\(\mu_n=|n|^2-c\) as in \eqref{eq:shifted-free-expansion}.
\begin{lemma}[High-frequency observability]
	\label{lem:high-frequency}
There exist constants \(N_0>0\) and \(\gamma_0>0\), depending only on
	\(d\), \(c\), and \(G\), such that
	\begin{equation}\label{4.10-7-28-wgs}
	\|F\|_{L^2(G)}^2
	\ge
	\gamma_0\|F\|_{L^2(Q)}^2,
	\qquad
	\forall F\in \mathcal X_{>N_0}.
	\end{equation}
\end{lemma}

\begin{proof}
Let \(N>0\) be fixed for the moment. Since finite sums are dense in
\(\mathcal X_{>N}\) and
\(
\|F\|_{L^2(G)}=\|\mathbf 1_GF\|_{L^2(Q)},
\)
with \(F\mapsto\mathbf 1_GF\) bounded on \(L^2(Q)\), it suffices to prove
the desired estimate uniformly for finite sums.

Let \(\phi,\psi\) be trigonometric polynomials on \(\mathbb T^d\) such that
\[
	\operatorname{supp}\widehat\phi
	\cup
	\operatorname{supp}\widehat\psi
	\subset
	\{n\in\mathbb Z^d:\ |n|>N\}.
\]
Define
\[
	u_+(t,x)=e^{ict}e^{it\Delta}\phi(x),
	\qquad
	u_-(t,x)=e^{-ict}e^{-it\Delta}\psi(x),
\]
and set
\(
	F=u_++u_-.
\)

Let \(C_{\rm obs}\) be the constant in \propref{prop:external-inputs}
(i). Choose \(\varepsilon>0\) so small that
\(
	3\varepsilon<C_{\rm obs}^{-1},
\)
and set
\(
	\alpha:=C_{\rm obs}^{-1}-\varepsilon>0.
\)
Let \(\zeta\) be the real-valued trigonometric polynomial given by
\propref{prop:external-inputs} (iii) for this \(\varepsilon\), and denote its
finite Fourier support by \(\mathcal Q_R\).

The proof is divided into five steps. We first establish separate lower
bounds for the two branches after multiplication by \(\zeta\). We then show
that their cross term tends to zero as the frequency cutoff tends to
infinity. After choosing \(N_0\) sufficiently large, we combine these
estimates to obtain a lower bound for \(\zeta F\), transfer it to
\(\chi F\) by the multiplier approximation, and finally pass from finite
spectral sums to the whole space \(\mathcal X_{>N_0}\).

\medskip
\noindent
\emph{Step 1. We prove
\begin{equation}\label{eq:branch-lower-bounds}
	\|\zeta u_+\|_{L^2(Q)}
	\ge
	\alpha\|u_+\|_{L^2(Q)},
	\qquad
	\|\zeta u_-\|_{L^2(Q)}
	\ge
	\alpha\|u_-\|_{L^2(Q)}.
\end{equation}}

Indeed, applying \propref{prop:external-inputs} (i), we have
\begin{equation}\label{eq:forward-branch-observability}
	\|u_+\|_{L^2(Q)}
	\le
	C_{\rm obs}\|\chi u_+\|_{L^2(Q)}.
\end{equation}
By \eqref{eq:forward-multiplier-estimate}, again using \(|e^{ict}|=1\),
\[
	\|(\chi-\zeta)u_+\|_{L^2(Q)}
	\le
	\varepsilon\|u_+\|_{L^2(Q)}.
\]
Combining this estimate with \eqref{eq:forward-branch-observability}, we get
\[
\begin{aligned}
	\|\zeta u_+\|_{L^2(Q)}
	&\ge
	\|\chi u_+\|_{L^2(Q)}
	-
	\|(\chi-\zeta)u_+\|_{L^2(Q)}
	\\
	&\ge
	(C_{\rm obs}^{-1}-\varepsilon)\|u_+\|_{L^2(Q)}
	=
	\alpha\|u_+\|_{L^2(Q)}.
\end{aligned}
\]

The second estimate follows in the same way from
\[
e^{-it\Delta}\psi
=
\overline{e^{it\Delta}\overline\psi},
\]
because \(\chi\) and \(\zeta\) are real-valued.
Thus \eqref{eq:branch-lower-bounds} is proved.
	
\medskip
\noindent
\emph{Step 2. We prove
\begin{equation}\label{eq:cross-term-small}
	|\langle \zeta u_+,\zeta u_-\rangle_{L^2(Q)}|
	\le
	\delta_N
	\|u_+\|_{L^2(Q)}
	\|u_-\|_{L^2(Q)},
	\qquad
	\delta_N\to0
	\quad\text{as }N\to\infty .
\end{equation}}

Write
\begin{equation}\label{4.14-7-28-w}
	\zeta(t,x)
	=
	\sum_{q\in\mathcal Q_R}
	\zeta_q e^{i(q_0t+q'\cdot x)},
	\qquad q=(q_0,q')\in\mathbb Z\times\mathbb Z^d ,
\end{equation}
and
\begin{equation}\label{4.15-7-28-w}
	u_+(t,x)
	=
	\sum_{|n|>N} a_n e^{i(n\cdot x-\mu_n t)},
	\qquad
	u_-(t,x)
	=
	\sum_{|m|>N} b_m e^{i(m\cdot x+\mu_m t)}.
\end{equation}
Using \eqref{4.14-7-28-w} and \eqref{4.15-7-28-w}, we obtain
\begin{equation}\label{4.16-7-28-w}
\begin{aligned}
	\langle \zeta u_+,\zeta u_-\rangle_{L^2(Q)}
	&=
	\sum_{q,r\in\mathcal Q_R}
	\zeta_q\overline{\zeta_r}
	\sum_{\substack{|n|>N\\ |m|>N}}
	a_n\overline{b_m}
	\\
	&\quad\times
	\int_0^{2\pi} e^{i(q_0-r_0-\mu_n-\mu_m)t}\,\d t
	\int_{\mathbb T^d}
	e^{i(n+q'-m-r')\cdot x}\,\d x .
\end{aligned}
\end{equation}
The spatial integral in \eqref{4.16-7-28-w} vanishes unless
\[
	m=n+q'-r'.
\]
Put \(h=q'-r'\). Since \(\mathcal Q_R\) is finite, there exists \(M_R>0\)
such that
\[
	|q_0-r_0|\le M_R,
	\qquad
	|q'-r'|\le M_R,
	\qquad q,r\in\mathcal Q_R.
\]
For \(m=n+h\), we have
\[
	\mu_n+\mu_{n+h}
	=
	2|n|^2+2n\cdot h+|h|^2-2c .
\]
Thus, by increasing \(N_0\) if necessary, we may assume that, for all
\(N\ge N_0\), all \(q,r\in\mathcal Q_R\), and all \(|n|>N\),
\[
	|q_0-r_0-\mu_n-\mu_{n+h}|
	\ge
	c_RN^2,
	\qquad h=q'-r',
\]
where \(c_R>0\) is independent of \(N\). Hence
\[
	\left|
	\int_0^{2\pi} e^{i(q_0-r_0-\mu_n-\mu_{n+h})t}\,\d t
	\right|
	\le
	C_RN^{-2}.
\]
Returning to \eqref{4.16-7-28-w} and using the Cauchy--Schwarz inequality, we get
\[
\begin{aligned}
|\langle \zeta u_+,\zeta u_-\rangle_{L^2(Q)}|
&\le
C_RN^{-2}
\sum_{q,r\in\mathcal Q_R}
|\zeta_q|\,|\zeta_r|
\sum_{|n|>N}|a_n|\,|b_{n+q'-r'}|
\\
&\le
\delta_N
\left(\sum_{|n|>N}|a_n|^2\right)^{1/2}
\left(\sum_{|m|>N}|b_m|^2\right)^{1/2},
\end{aligned}
\]
where \(\delta_N\to0\) as \(N\to\infty\). Since \eqref{4.15-7-28-w} gives
\[
	\|u_+\|_{L^2(Q)}^2
	=
	2\pi\sum_{|n|>N}|a_n|^2,
	\qquad
	\|u_-\|_{L^2(Q)}^2
	=
	2\pi\sum_{|m|>N}|b_m|^2,
\]
we obtain \eqref{eq:cross-term-small}.

\medskip
\noindent
\emph{Step 3. We prove
\begin{equation}\label{eq:zetaF-lower-bound}
	\|\zeta F\|_{L^2(Q)}^2
	\ge
	c_1
	\left(
	\|u_+\|_{L^2(Q)}^2
	+
	\|u_-\|_{L^2(Q)}^2
	\right),
\end{equation}
where \(c_1>0\) is independent of \(N\) and \(F\).}

Indeed,
\[
\begin{aligned}
	\|\zeta F\|_{L^2(Q)}^2
	&=
	\|\zeta u_+\|_{L^2(Q)}^2
	+
	\|\zeta u_-\|_{L^2(Q)}^2
	+
	2\operatorname{Re}
	\langle \zeta u_+,\zeta u_-\rangle_{L^2(Q)}.
\end{aligned}
\]
This, along with \eqref{eq:branch-lower-bounds} and
\eqref{eq:cross-term-small}, yields
\begin{equation}\label{4.18-2-28-w}
\begin{aligned}
	\|\zeta F\|_{L^2(Q)}^2
	\ge
	\alpha^2
	\left(
	\|u_+\|_{L^2(Q)}^2
	+
	\|u_-\|_{L^2(Q)}^2
	\right)
	-
	2\delta_N
	\|u_+\|_{L^2(Q)}
	\|u_-\|_{L^2(Q)} .
\end{aligned}
\end{equation}
Since
\[
	2\|u_+\|_{L^2(Q)}\|u_-\|_{L^2(Q)}
	\le
	\|u_+\|_{L^2(Q)}^2
	+
	\|u_-\|_{L^2(Q)}^2,
\]
we increase \(N_0\) if necessary so that
\[
	\delta_N\le \alpha^2/2,
	\qquad N\ge N_0.
\]
Then, by \eqref{4.18-2-28-w}, we have
\[
	\|\zeta F\|_{L^2(Q)}^2
	\ge
	\frac{\alpha^2}{2}
	\left(
	\|u_+\|_{L^2(Q)}^2
	+
	\|u_-\|_{L^2(Q)}^2
	\right),
\]
which gives \eqref{eq:zetaF-lower-bound} with \(c_1=\alpha^2/2\).

\medskip
\noindent
\emph{Step 4. We prove
\begin{equation}\label{eq:chiF-lower-bound}
	\|\chi F\|_{L^2(Q)}^2
	\ge
	c_2
	\left(
	\|u_+\|_{L^2(Q)}^2
	+
	\|u_-\|_{L^2(Q)}^2
	\right),
\end{equation}
for some \(c_2>0\).}

By \eqref{eq:shifted-double-branch-estimate},
\begin{equation}\label{4.20-7-28-w}
\begin{aligned}
\|(\chi-\zeta)F\|_{L^2(Q)}
&\le
\varepsilon
\left(
\|u_+\|_{L^2(Q)}
+
\|u_-\|_{L^2(Q)}
\right)
\\
&\le
\sqrt{2}\,\varepsilon
\left(
\|u_+\|_{L^2(Q)}^2
+
\|u_-\|_{L^2(Q)}^2
\right)^{1/2}.
\end{aligned}
\end{equation}
Since \(\chi F=\zeta F+(\chi-\zeta)F\), we have
\[
	\|\chi F\|_{L^2(Q)}
	\ge
	\|\zeta F\|_{L^2(Q)}
	-
	\|(\chi-\zeta)F\|_{L^2(Q)}.
\]
Combining this with \eqref{4.20-7-28-w} and
\eqref{eq:zetaF-lower-bound}, we obtain
\begin{equation}\label{4.21-7-28-w}
\|\chi F\|_{L^2(Q)}
\ge
\left(\sqrt{c_1}-\sqrt{2}\,\varepsilon\right)
\left(
\|u_+\|_{L^2(Q)}^2
+
\|u_-\|_{L^2(Q)}^2
\right)^{1/2}.
\end{equation}
Since \(c_1=\alpha^2/2\), \(\alpha=C_{\rm obs}^{-1}-\varepsilon\), and
\(3\varepsilon<C_{\rm obs}^{-1}\), we have
\[
	\sqrt{c_1}-\sqrt{2}\,\varepsilon
	=
	\frac{C_{\rm obs}^{-1}-3\varepsilon}{\sqrt{2}}
	>
	0.
\]
Therefore \eqref{4.21-7-28-w} gives \eqref{eq:chiF-lower-bound} with
\[
	c_2
	=
	\left(\sqrt{c_1}-\sqrt{2}\,\varepsilon\right)^2
	>
	0.
\]

\medskip
\noindent
\emph{Step 5. We prove \eqref{4.10-7-28-wgs}.}

Since
\[
	\|F\|_{L^2(Q)}^2
	=
	\|u_+ + u_-\|_{L^2(Q)}^2
	\le
	2\|u_+\|_{L^2(Q)}^2
	+
	2\|u_-\|_{L^2(Q)}^2,
\]
and
\[
	\|\chi F\|_{L^2(Q)}
	=
	\|F\|_{L^2(G)},
\]
we get from \eqref{eq:chiF-lower-bound} that
\[
	\|F\|_{L^2(G)}^2
	\ge
	\frac{c_2}{2}
	\|F\|_{L^2(Q)}^2.
\]
Taking \(N=N_0\), and using the density of finite sums in
\(\mathcal X_{>N_0}\) together with the boundedness of multiplication by
\(\chi\) on \(L^2(Q)\), we obtain \eqref{4.10-7-28-wgs} with
\(\gamma_0=c_2/2\). Here \(C_{\rm obs}\) and the multiplier \(\zeta\) depend only on
\(d\), \(G\), and the chosen \(\varepsilon\), while the high-frequency
threshold additionally depends on \(c\). Hence \(N_0\) and \(\gamma_0\)
depend only on \(d\), \(c\), and \(G\).

This completes the proof of the lemma.
			\end{proof}

\subsection{Product stability under small removals}
\label{subsec:product-stability}

We shall use the following elementary stability fact in the low-frequency
absorption argument: observability on an admissible measurable product set
remains valid after removing small subsets from the product factors, provided
the closed subspace under consideration satisfies suitable mixed-norm bounds.

Let
\[
G=G_1\times\cdots\times G_m\subset Q
\]
be an admissible measurable product set as in \eqref{1.10-7-27-w}.
For each \(j\), set
\[
Q_{\widehat{\alpha_j}}
:=
\prod_{\ell\notin\alpha_j}Q_\ell .
\]
The variables \(z_{\alpha_j}\) and \(z_{\widehat{\alpha_j}}\) are defined
in \eqref{eq:block-variable-definition} and \eqref{4.2-4-30-w},
respectively.

Let \(X\) be a closed subspace of \(L^2(Q)\). We assume that \(X\) satisfies
the following hypothesis.

\medskip
\noindent
\textbf{Hypothesis (H).}
There exist constants \(\gamma>0\) and \(K_X>0\) such that the following hold.

\begin{enumerate}
	\item Observability on \(G\):
	\[
	\|f\|_{L^2(G)}^2
	\ge
	\gamma \|f\|_{L^2(Q)}^2,
	\qquad \forall f\in X.
	\]
	
	\item Mixed-norm estimates:
	\[
	\|f\|_{L^4_{z_{\alpha_j}}L^2_{z_{\widehat{\alpha_j}}}}
	\le
	K_X\|f\|_{L^2(Q)},
	\qquad \forall f\in X,\quad j=1,\ldots,m.
	\]
\end{enumerate}

\begin{lemma}[Product stability under small removals]
	\label{lem:product-stability}
	With the above notation and Hypothesis (H), there exists \(\delta>0\),
	depending only on \(\gamma\), \(K_X\), and \(G\), such that if
	\(I_j\subset G_j\) is measurable and \(|G_j\setminus I_j|\le\delta\) for every \(j\), then
	\[
	\|f\|_{L^2(I)}^2
	\ge
	\frac{\gamma}{2}\|f\|_{L^2(Q)}^2,
	\qquad \forall f\in X.
	\]
	Here \(I=I_1\times\cdots\times I_m\).
\end{lemma}

\begin{proof}
	Fix \(f\in X\). For each \(j=1,\ldots,m\), define
	\[
	H_j(z_{\alpha_j})
	=
	\int_{\prod_{i\ne j}G_i}
	|f(z_{\alpha_j},z_{\widehat{\alpha_j}})|^2\,
	\d z_{\widehat{\alpha_j}} .
	\]
	
	We first estimate the loss caused by removing a small subset from a single
	factor. Since \(\prod_{i\ne j}G_i
\subset Q_{\widehat{\alpha_j}}\), we have
	\[
	H_j(z_{\alpha_j})
	\le
	\int |f(z_{\alpha_j},z_{\widehat{\alpha_j}})|^2\,
	\d z_{\widehat{\alpha_j}} .
	\]
	Therefore, by (2) of Hypothesis (H), for every \(j=1,\ldots,m\),
	\[
	\|H_j\|_{L^2(G_j)}
	\le
	\|f\|_{L^4_{z_{\alpha_j}}
	L^2_{z_{\widehat{\alpha_j}}}}^2
	\le
	K_X^2\|f\|_{L^2(Q)}^2.
	\]
	Hence, by the Cauchy--Schwarz inequality,
	\begin{equation}\label{4.23-7-28-w}
	\begin{aligned}
	\int_{(G_j\setminus I_j)\times\prod_{i\ne j}G_i}|f|^2
	&=
	\int_{G_j\setminus I_j}H_j(z_{\alpha_j})\,\d z_{\alpha_j}
	\\
	&\le
	|G_j\setminus I_j|^{1/2}\|H_j\|_{L^2(G_j)}
	\\
	&\le
	\delta^{1/2}K_X^2\|f\|_{L^2(Q)}^2.
	\end{aligned}
	\end{equation}
	
	We now pass from one factor to the whole product. Since
	\[
	G\setminus I
	\subset
	\bigcup_{j=1}^m
	\left[
	(G_j\setminus I_j)\times\prod_{i\ne j}G_i
	\right],
	\]
	it follows from \eqref{4.23-7-28-w} that
	\[
	\|f\|_{L^2(G\setminus I)}^2
	\le
	m\delta^{1/2}K_X^2\|f\|_{L^2(Q)}^2.
	\]
	Choose
	\[
	\delta
	=
	\min\left\{
	1, \frac12\min_{1\leq j\leq m}|G_j|,
	\left(\frac{\gamma}{2mK_X^2}\right)^2
	\right\}.
	\]
	Then
	\begin{equation}\label{4.24-7-28-w}
	\|f\|_{L^2(G\setminus I)}^2
	\le
	\frac{\gamma}{2}\|f\|_{L^2(Q)}^2.
	\end{equation}
	
	Finally, since \(I\subset G\), by \eqref{4.24-7-28-w} and (1) of
	Hypothesis (H),
	\[
	\|f\|_{L^2(I)}^2
	=
	\|f\|_{L^2(G)}^2
	-
	\|f\|_{L^2(G\setminus I)}^2
	\ge
	\frac{\gamma}{2}\|f\|_{L^2(Q)}^2.
	\]
	This completes the proof.
\end{proof}

\subsection{Spectral structure and translations on the double spectrum}
\label{sec:double-spectrum}

The purpose of this subsection is to describe the spectral structure of the
double spectrum \(\Sigma_c\) defined in \eqref{1.18-7-28-w}, and the
translation properties needed later to pass from the high-frequency estimate
on \(\mathcal X_{>N_0}\) in \lemref{lem:high-frequency} to the full
double-spectrum space \(\mathcal X\).

Let \(c\in\mathbb R\) be fixed as in \thmref{thm:full}, and set
\(\mu_n=|n|^2-c\) for \(n\in\mathbb Z^d\). We write the double spectrum
\(\Sigma_c\) in \eqref{1.18-7-28-w} in the time-space order
\[
	\Sigma_c
	=
	\{(-\mu_n,n),(\mu_n,n): n\in\mathbb Z^d\}
	\subset \mathbb R\times\mathbb Z^d .
\]
For \(\eta=(\tau,k)\in\Sigma_c\), write
\begin{equation}\label{4.25-7-29-w}
	e_\eta(t,x)=e^{i(\tau t+k\cdot x)},
	\qquad t\in(0,2\pi),\quad x\in\mathbb T^d .
\end{equation}
For \(\Lambda\subset\Sigma_c\), let \(X_\Lambda\) be the closure in \(L^2(Q)\)
of the linear span of \(\{e_\eta:\eta\in\Lambda\}\).

For each \(n\in\mathbb Z^d\), let \(E_n^\Lambda\) be the subspace of
\(\operatorname{span}\{e^{-i\mu_n t},e^{i\mu_n t}\}\) spanned by the time
exponentials whose corresponding spacetime frequencies
\[
	(-\mu_n,n),\qquad (\mu_n,n)
\]
belong to \(\Lambda\). When \(\mu_n=0\), the two generators coincide and are
counted only once. Then every \(f\in X_\Lambda\) admits the spatial Fourier
expansion
\[
	f(t,x)=\sum_{n\in\mathbb Z^d} e^{in\cdot x}f_n(t),
	\qquad
	f_n\in E_n^\Lambda,
\]
with convergence in \(L^2(Q)\).

The time-frequency blocks \(E_n^\Lambda\) satisfy uniform Riesz bounds. More
precisely, there exist constants \(0<A_c\le B_c<\infty\), depending only on
\(c\), such that, for every \(n\in\mathbb Z^d\) and every
\(f_n\in E_n^\Lambda\),
\begin{equation}\label{equ-Riesz-1}
	A_c|\alpha_n|^2
	\le
	\|f_n\|_{L^2(0,2\pi)}^2
	\le
	B_c|\alpha_n|^2 .
\end{equation}
Here \(\alpha_n\) denotes the coefficient vector of \(f_n\) with respect to
the distinct generators of \(E_n^\Lambda\) selected from
\(\{e^{-i\mu_n t},e^{i\mu_n t}\}\), with the convention \(\alpha_n=0\) if
\(E_n^\Lambda=\{0\}\).

Indeed, this follows by inspecting the Gram matrices of size at most
\(2\times2\). When both generators are present and \(\mu_n\neq0\), their Gram matrix is
\[
\mathsf G_n
=
\begin{pmatrix}
	2\pi & J_n\\
	\overline{J_n} & 2\pi
\end{pmatrix},
\qquad
J_n
:=
\int_0^{2\pi}e^{-2i\mu_nt}\,\d t.
\]
Its eigenvalues are \(2\pi\pm|J_n|\). Since \(\mu_n\neq0\), the two
exponentials are linearly independent and hence \(|J_n|<2\pi\). Moreover,
the nonzero values of \(\mu_n=|n|^2-c\) form a discrete set separated from
zero, while \(J_n\to0\) as \(|n|\to\infty\). Therefore
\[
\sup_{\mu_n\neq0}|J_n|<2\pi,
\]
which yields uniform positive lower and upper bounds for all two-generator
blocks. The one-generator blocks are immediate.

As a direct consequence of \eqref{equ-Riesz-1} and the orthogonality of
spatial Fourier modes,
\begin{equation}\label{equ-Riesz-2}
	\|f\|_{L^2(Q)}^2
	\simeq_{c,d}
	\sum_{n\in\mathbb Z^d}|\alpha_n|^2 .
\end{equation}

The next lemma records the boundedness and local translation property of
spectral shifts on the spaces \(X_\Lambda\) associated with subsets
\(\Lambda\subset\Sigma_c\). It will be used later to form finite-difference
operators in the low-frequency absorption argument.
 
\begin{lemma}[Bounded spectral translations]
	\label{lem:spectral-translations}
	Let \(\Lambda\subset\Sigma_c\), let \(X_\Lambda\) be as above, and let
	\(\xi=(s,y)\in\mathbb R\times\mathbb R^d\). Define \(T_\xi\) on finite
	spectral sums by
	\[
	T_\xi e_\eta
	=
	e^{-i\eta\cdot\xi}e_\eta,
	\qquad
	\eta\in\Lambda .
	\]
	Then the following hold.
	\begin{itemize}
		\item[(i)] \(T_\xi\) extends uniquely to a bounded operator on
		\(X_\Lambda\), and
		\[
		\|T_\xi\|_{\mathcal L(X_\Lambda)}
		\le C_{c,d},
		\]
		where \(C_{c,d}\) is independent of \(\Lambda\) and \(\xi\).

		\item[(ii)] If \(\Omega\subset Q\) is measurable and
		\(\Omega-\xi\subset Q\), then for every \(f\in X_\Lambda\),
		\[
		T_\xi f=f(\cdot-\xi)
		\quad\text{in }L^2(\Omega),
		\]
		with spatial variables understood modulo \(2\pi\).
	\end{itemize}
\end{lemma}

\begin{proof}
	(i) Write a finite spectral sum \(p\in X_\Lambda\) as
\[
p(t,x)=\sum_n e^{in\cdot x}p_n(t),\qquad p_n\in E_n^\Lambda .
\]
On each time-frequency block \(E_n^\Lambda\), the operator \(T_\xi\)
multiplies the basis coefficients by unimodular factors, and therefore
preserves the \(\ell^2\)-norm of the coefficient vector. By the uniform
Riesz bounds \eqref{equ-Riesz-1}, we have
\[
\|T_\xi p\|_{L^2(Q)}\le C_{c,d}\|p\|_{L^2(Q)}
\]
for all such finite sums. Since finite spectral sums are dense in \(X_\Lambda\), \(T_\xi\) extends
uniquely to a bounded operator on \(X_\Lambda\) with the same bound.
	
(ii) Let \(\Omega\subset Q\) be measurable and satisfy
\(\Omega-\xi\subset Q\). For a finite spectral sum \(p\), the identity
\[
	T_\xi p=p(\cdot-\xi)
\]
holds in \(L^2(\Omega)\), because
\[
	e_\eta(z-\xi)=e^{-i\eta\cdot\xi}e_\eta(z).
\]
Now let \(f\in X_\Lambda\), and choose finite spectral sums
\(p_k\to f\) in \(L^2(Q)\). By (i),
\[
	T_\xi p_k\to T_\xi f
	\quad\text{in }L^2(\Omega).
\]
Moreover, since \(\Omega-\xi\subset Q\),
\[
	\|p_k(\cdot-\xi)-f(\cdot-\xi)\|_{L^2(\Omega)}
 =
	\|p_k-f\|_{L^2(\Omega-\xi)} \leq
	\|p_k-f\|_{L^2(Q)}
	\longrightarrow0.
\]
Passing to the limit gives
\[
	T_\xi f=f(\cdot-\xi)
	\quad\text{in }L^2(\Omega).
\]
\end{proof}

\subsection{Absorption of a single spacetime frequency}
\label{subsec:single-frequency-absorption}

Throughout this subsection, we use the notation and the time-space ordering of
\(\Sigma_c\) from the preceding subsection. Let \(\Lambda\subset\Sigma_c\), and
set
\[
	Y:=X_\Lambda .
\]
By \lemref{lem:spectral-translations}, for every
\(\xi\in\mathbb R\times\mathbb R^d\), one has
\(
	T_\xi\in\mathcal L(Y).
\)

Assume that \(Y\) satisfies Hypothesis (H) from
\subsecref{subsec:product-stability}, with constants \(\gamma>0\) and
\(K_Y>0\). That is, 
\begin{equation}\label{eq:old-space-observable}
	\|f\|_{L^2(G)}^2
	\ge
	\gamma\|f\|_{L^2(Q)}^2,
	\qquad
	f\in Y .
\end{equation}

We first record a simple way to choose a non-resonant small translation. This
will be used in the proof of \lemref{lem:one-frequency-absorption}.

\begin{lemma}[Non-resonant small translations]
	\label{lem:nonresonant-translation-choice}
	With the above notation, let \(\lambda\in\Sigma_c\setminus\Lambda\).
	Then, for every \(\rho>0\), there exists
	\[
		\xi\in B_{\mathbb R^{1+d}}(0,\rho)
	\]
	such that
	\begin{equation}\label{eq:nonresonant-xi-choice}
		e^{-i\eta\cdot\xi}
		\neq
		e^{-i\lambda\cdot\xi},
		\qquad
		\forall \eta\in\Lambda .
	\end{equation}
\end{lemma}
\begin{proof}
	Let \(\rho>0\) be fixed. Since \(\lambda\notin\Lambda\), for every
	\(\eta\in\Lambda\) we have \(\eta-\lambda\neq0\).
	For fixed \(\eta\in\Lambda\), the resonance condition
	\[
		e^{-i\eta\cdot\xi}
		=
		e^{-i\lambda\cdot\xi}
	\]
	is equivalent to
	\[
		(\eta-\lambda)\cdot\xi\in 2\pi\mathbb Z .
	\]
	Thus its solution set is a countable union of proper affine hyperplanes in
	\(\mathbb R^{1+d}\). Since \(\Lambda\subset\Sigma_c\) and \(\Sigma_c\) is
	countable, the union of these exceptional sets over all \(\eta\in\Lambda\)
	has Lebesgue measure zero. Hence it cannot contain
	\(B_{\mathbb R^{1+d}}(0,\rho)\).

	We choose \(\xi\in B_{\mathbb R^{1+d}}(0,\rho)\) outside this exceptional
	set. Then \eqref{eq:nonresonant-xi-choice} holds, and the proof is complete.
\end{proof}

The next lemma is the basic absorption step. It shows that observability is
preserved after adding one new spacetime frequency to \(Y=X_\Lambda\). The
finite-difference argument used in its proof is inspired by the proof of
Miheev's theorem in the uncertainty principle for trigonometric series; see
Havin--J\"oricke~\cite{HavinJoricke1994} and the survey of
Bonami--Demange~\cite{BonamiDemange2006}.

\begin{lemma}[Absorption of a single spacetime frequency]
	\label{lem:one-frequency-absorption}
	With the above notation and assumptions, let
	\(\lambda\in\Sigma_c\setminus\Lambda\), and let \(e_\lambda\) be defined
	by \eqref{4.25-7-29-w}. Then there exists \(\gamma'>0\) such that
\begin{align}\label{equ84-1}
	\|g+ae_\lambda\|_{L^2(G)}^2
\ge
\gamma'
\|g+ae_\lambda\|_{L^2(Q)}^2,
\qquad
\forall g\in Y,\quad \forall a\in\mathbb C .
\end{align}
\end{lemma}
\begin{proof}
If $e_\lambda\in Y$, then \eqref{equ84-1} holds clearly. So we assume $e_\lambda\notin Y$ now.
We argue by contradiction. The proof has two main ideas. First, from the
failure of the estimate we obtain a normalized limit \(h=f+ae_\lambda\) which
vanishes on \(G\). Second, we use a small non-resonant translation to form a finite difference
which cancels the \(e_\lambda\)-term. Product stability and the
non-resonance condition then give the contradiction.

We give the details in the following steps.
	
\medskip
\noindent
\emph{Step 1. We construct a nonzero limit from a contradiction sequence.}

Assume that \eqref{equ84-1} is false. Then, there exist sequences
\(\{f_\nu\}_{\nu\ge1}\subset Y\) and
\(\{a_\nu\}_{\nu\ge1}\subset\mathbb C\) such that
\begin{equation}\label{eq:absorption-contradiction-sequence}
	\|f_\nu+a_\nu e_\lambda\|_{L^2(Q)}=1,
	\qquad
	\|f_\nu+a_\nu e_\lambda\|_{L^2(G)}\to0
	\quad\text{as } \nu\to\infty .
\end{equation}
Since $e_\lambda\notin Y$ and \(Y\) is closed in
\(L^2(Q)\), it follows that
\[
	d_\lambda:=\operatorname{dist}(e_\lambda,Y)>0 .
\]
Together with the first equality in
\eqref{eq:absorption-contradiction-sequence}, this gives
\[
	1
	=
	\|f_\nu+a_\nu e_\lambda\|_{L^2(Q)}
	\ge
	|a_\nu|d_\lambda,
	\qquad \nu\ge1.
\]
Hence \(\{a_\nu\}_{\nu\ge1}\) is bounded. Passing to a subsequence, we may
assume that
\begin{equation}\label{4.31-7-29-wg}
	a_\nu\to a
	\quad\text{as } \nu\to\infty .
\end{equation}

We next prove that \(\{f_\nu\}_{\nu\ge1}\) is Cauchy in \(L^2(Q)\). Since
\(Y\) is a linear subspace, \(f_\nu-f_\ell\in Y\). By
\eqref{eq:old-space-observable},
\begin{equation}\label{4.32-7-29-wg}
	\|f_\nu-f_\ell\|_{L^2(Q)}
	\le
	\gamma^{-1/2}\|f_\nu-f_\ell\|_{L^2(G)} .
\end{equation}
On the other hand,
\[
	f_\nu-f_\ell
	=
	(f_\nu+a_\nu e_\lambda)
	-
	(f_\ell+a_\ell e_\lambda)
	-
	(a_\nu-a_\ell)e_\lambda \;\;\mbox{ in } L^2(G).
\]
The first two terms tend to zero in \(L^2(G)\) as
\(\nu,\ell\to\infty\), by
\eqref{eq:absorption-contradiction-sequence}, and the last term tends to zero
by \eqref{4.31-7-29-wg}. Hence
\[
	\|f_\nu-f_\ell\|_{L^2(G)}\to0
	\quad\text{as } \nu,\ell\to\infty .
\]
This, together with \eqref{4.32-7-29-wg}, shows that
\(\{f_\nu\}_{\nu\ge1}\) is Cauchy in \(L^2(Q)\). Since \(Y\) is closed, there
exists \(f\in Y\) such that
\[
	f_\nu\to f
	\quad\text{in }L^2(Q).
\]
Set
\begin{equation}\label{4.34-7-29-wwggss}
	h=f+ae_\lambda .
\end{equation}
Passing to the limit in \eqref{eq:absorption-contradiction-sequence}, we
obtain
\begin{equation}\label{eq:limit-invisible-function}
	\|h\|_{L^2(Q)}=1,
	\qquad
	h=0\quad\text{a.e. on }G .
\end{equation}

\medskip
\noindent
\emph{Step 2. We choose an interior product set on which small translations remain inside \(G\).}

Let \(\delta>0\) be the constant in \lemref{lem:product-stability}, applied
to the space \(Y\). By regularity of Lebesgue measure, choose compact sets
\(K_j\subset G_j\) such that
\begin{equation}\label{4.34-7-29-ww}
	|G_j\setminus K_j|<\delta/4,
	\qquad j=1,\ldots,m .
\end{equation}
For the factor containing the time variable, we choose \(K_j\Subset
Q_{\alpha_j}\), where \(Q_{\alpha_j}\) is defined after
\eqref{1.10-7-27-w}.

For \(\theta=(\theta_0,\ldots,\theta_d)\in\mathbb R^{1+d}\), write
\[
	\theta_{\alpha_j}:=(\theta_\ell)_{\ell\in\alpha_j},
	\qquad j=1,\ldots,m .
\]
By continuity of translations in measure, there exists \(\rho>0\) such that,
whenever \(|\theta|<\rho\),
\begin{equation}\label{eq:small-translation-product-prep}
	|K_j\setminus(K_j+\theta_{\alpha_j})|<\delta/4,
	\qquad j=1,\ldots,m .
\end{equation}

By \lemref{lem:nonresonant-translation-choice}, we can choose
\[\xi\in B_{\mathbb R^{1+d}}(0,\rho)\] such that
\eqref{eq:nonresonant-xi-choice} holds. Define
\[
	I_j:=K_j\cap(K_j+\xi_{\alpha_j}),
	\qquad
	I:=\prod_{j=1}^m I_j .
\]
These, along with \eqref{4.34-7-29-ww} and
\eqref{eq:small-translation-product-prep} applied to \(\theta=\xi\), yield
\begin{equation}\label{eq:localized-product-set-I}
	|G_j\setminus I_j|<\delta,
	\qquad
	I\subset G,
	\qquad
	I-\xi\subset G .
\end{equation}

\medskip
\noindent
\emph{Step 3. We construct a difference operator which cancels the \(e_\lambda\)-term.}

For the \(\xi\) chosen in Step~2, let \(T_\xi\) be the bounded operator on
\(Y\) given by \lemref{lem:spectral-translations}.
 Since \(e_\lambda\notin Y\), every element
of \(Y+\operatorname{span}\{e_\lambda\}\) has a unique form
\(g+ae_\lambda\), where \(g\in Y\) and \(a\in\mathbb C\). We extend \(T_\xi\)
to \(Y+\operatorname{span}\{e_\lambda\}\) by
\begin{equation}\label{eq:Txi-extension-on-augmented-space}
	T_\xi(g+ae_\lambda)
	=
	T_\xi g
	+
	a e^{-i\lambda\cdot\xi}e_\lambda,
	\qquad
	g\in Y,\quad a\in\mathbb C .
\end{equation}
Define
\begin{equation}\label{eq:difference-operator-Dxi}
	D_\xi
	=
	T_\xi
	-
	e^{-i\lambda\cdot\xi}
	\operatorname{Id}_{Y+\operatorname{span}\{e_\lambda\}}
	\quad\text{on }Y+\operatorname{span}\{e_\lambda\}.
\end{equation}
By \eqref{4.34-7-29-wwggss}, \eqref{eq:Txi-extension-on-augmented-space} and
\eqref{eq:difference-operator-Dxi}, we have
\begin{equation}\label{eq:Dxi-h-belongs-to-Y}
	D_\xi e_\lambda=0,
	\qquad
	D_\xi h=D_\xi f\in Y .
\end{equation}

By \eqref{eq:localized-product-set-I} and \(G\subset Q\), we have
\(I\subset Q\) and \(I-\xi\subset Q\). Applying
\lemref{lem:spectral-translations} with \(\Omega=I\), we obtain
\begin{equation}\label{eq:Txi-f-local-realization}
	T_\xi f=f(\cdot-\xi)
	\quad\text{in }L^2(I).
\end{equation}
By \eqref{eq:Txi-extension-on-augmented-space} and \eqref{4.25-7-29-w}, we have
\[
	T_\xi e_\lambda=e_\lambda(\cdot-\xi)
	\quad\text{in }L^2(I).
\]
This, along with \eqref{4.34-7-29-wwggss} and
\eqref{eq:Txi-f-local-realization}, yields
\[
	T_\xi h=h(\cdot-\xi)
	\quad\text{in }L^2(I).
\]
The above, together with \eqref{eq:limit-invisible-function} and
\eqref{eq:localized-product-set-I}, implies
\[
	h=0,
	\qquad
	T_\xi h=0
	\quad\text{in }L^2(I).
\]
Combining this with \eqref{eq:difference-operator-Dxi} and
\eqref{eq:Dxi-h-belongs-to-Y}, we obtain
\begin{equation}\label{eq:Dxi-f-vanishes-on-I}
	D_\xi f
	=
	D_\xi h
	=
	0
	\quad\text{in }L^2(I).
\end{equation}

\medskip
\noindent
\emph{Step 4. We use product stability to force the difference to vanish globally.}

By \eqref{eq:Dxi-h-belongs-to-Y}, we have \(D_\xi f\in Y\). Applying
\lemref{lem:product-stability} to \(D_\xi f\), and using
\eqref{eq:localized-product-set-I}, we get
\begin{equation}\label{eq:product-stability-Dxi-f}
	\frac{\gamma}{2}\|D_\xi f\|_{L^2(Q)}^2
	\le
	\|D_\xi f\|_{L^2(I)}^2 .
\end{equation}
By \eqref{eq:Dxi-f-vanishes-on-I} and
\eqref{eq:product-stability-Dxi-f}, we obtain
\begin{equation}\label{eq:Dxi-f-global-zero}
	D_\xi f=0
	\quad\text{in }L^2(Q).
\end{equation}
	
	\medskip
\noindent
\emph{Step 5. We derive a contradiction.}

Since \(f\in Y=X_\Lambda\), write
\[
	f=\sum_{\eta\in\Lambda}c_\eta e_\eta
	\quad\text{in }L^2(Q).
\]
By \eqref{eq:difference-operator-Dxi} and \eqref{eq:Dxi-f-global-zero}, we have
\[
	\sum_{\eta\in\Lambda}
	\bigl(e^{-i\eta\cdot\xi}-e^{-i\lambda\cdot\xi}\bigr)c_\eta e_\eta
	=
	0
	\quad\text{in }L^2(Q).
\]
By uniqueness of the spectral expansion and
\eqref{eq:nonresonant-xi-choice}, it follows that \(c_\eta=0\) for every
\(\eta\in\Lambda\). Hence
\[
	f=0\quad\text{in }L^2(Q).
\]
By \eqref{4.34-7-29-wwggss}, we obtain \(h=ae_\lambda\). Since \(h=0\) a.e.
on \(G\) by \eqref{eq:limit-invisible-function}, while \(|e_\lambda|=1\) and
\(|G|>0\), we have \(a=0\). Thus \(h=0\), contradicting the first equality in
\eqref{eq:limit-invisible-function}. This proves the lemma.
\end{proof}

\subsection[Proof of the double-spectrum observability theorem]{Proof of \texorpdfstring{\thmref{thm:double-spectrum-obs}}{the double-spectrum observability theorem}}

Let \(N_0\) and \(\gamma_0\) be given by
\lemref{lem:high-frequency}. Increasing \(N_0\), if necessary, we may therefore assume that
\[
|n|>N_0 \quad\Longrightarrow\quad \mu_n\neq0.
\]
Hence the two generators in each high-frequency block are distinct. Define
\[
X_0:=\mathcal X_{>N_0},
\]
where \(\mathcal X_{>N}\) is defined in
\subsecref{subsec:measurable-high-frequency}. By  \lemref{lem:high-frequency}, we have
\[
	\|F\|_{L^2(G)}^2
	\ge
	\gamma_0\|F\|_{L^2(Q)}^2,
	\qquad
	F\in X_0.
\]
Moreover, by the definitions of
\(X_\Lambda\) and \(\mathcal X_{>N}\) in
\subsecref{sec:double-spectrum} and
\subsecref{subsec:measurable-high-frequency}, respectively, we have
\[
	X_0=X_{\Lambda_{>N_0}},
\]
where
\[
	\Lambda_{>N_0}
	=
	\{(-\mu_n,n),(\mu_n,n): |n|>N_0\}.
\]

The remaining low-frequency set is finite:
\begin{equation}\label{eq:low-frequency-set}
	\Lambda_{\le N_0}
	=
	\{(-\mu_n,n),(\mu_n,n): |n|\le N_0\}.
\end{equation}
After deleting possible repetitions, write
\[
	\Lambda_{\le N_0}
	=
	\{\lambda_1,\ldots,\lambda_M\}.
\]

We now absorb these finitely many low frequencies one by one. Define
recursively
\begin{equation*}
	\Lambda_j
	=
	\Lambda_{>N_0}\cup\{\lambda_1,\ldots,\lambda_j\},
	\qquad
	X_j=X_{\Lambda_j},
	\qquad
	j=1,\ldots,M .
\end{equation*}
Thus
\begin{equation}\label{4.48-7-31-w}
	\lambda_j\notin\Lambda_{j-1},
	\qquad
	X_j
	=
	X_{j-1}\oplus \operatorname{span}\{e_{\lambda_j}\},
	\qquad j=1,\ldots,M .
\end{equation}

We next explain why \lemref{lem:one-frequency-absorption} can be applied
successively. The initial space \(X_0\) satisfies the observability estimate by
\eqref{4.10-7-28-wgs}. It also satisfies the mixed-norm estimates in Hypothesis (H).
For a finite sum \(F=u_++u_-\in X_0\), write
\[
u_+=e^{ict}e^{it\Delta}\phi,
\qquad
u_-=e^{-ict}e^{-it\Delta}\psi.
\]
Then part~(ii) of \propref{prop:external-inputs}
(and its backward analogue obtained by complex conjugation),
the triangle inequality, and \eqref{equ-Riesz-2} give, for every \(j\),
\[
\begin{aligned}
	\|F\|_{L^4_{z_{\alpha_j}}L^2_{z_{\widehat{\alpha_j}}}}
	&\le
	C_Z\bigl(\|\phi\|_{L^2_x}+\|\psi\|_{L^2_x}\bigr)
	\\
	&\le
	C_{c,d}\|F\|_{L^2(Q)}.
\end{aligned}
\]
Since
\(L^4_{z_{\alpha_j}}L^2_{z_{\widehat{\alpha_j}}}\)
is reflexive and continuously embedded in \(L^2(Q)\), weak compactness
and lower semicontinuity extend the estimate from finite sums to every
\(F\in X_0\).

Suppose that \(X_{j-1}\) satisfies Hypothesis (H). It suffices to consider the case
\[
	e_{\lambda_j}\notin X_{j-1}.
\]
Since \(X_{j-1}\) is closed in \(L^2(Q)\), the distance from
\(e_{\lambda_j}\) to \(X_{j-1}\) is positive. Hence there exists \(C>0\) such
that, for every decomposition
\[
	u=v+ae_{\lambda_j},
	\qquad
	v\in X_{j-1},\quad a\in\mathbb C,
\]
one has
\[
	\|v\|_{L^2(Q)}+|a|
	\le
	C\|u\|_{L^2(Q)} .
\]
Since all mixed norms on the one-dimensional space
\(\operatorname{span}\{e_{\lambda_j}\}\) are equivalent to its
\(L^2(Q)\)-norm, the mixed-norm estimates in Hypothesis (H) remain valid for
\(X_j\), possibly with a larger constant. On the other hand,
\lemref{lem:one-frequency-absorption} gives the observability estimate on
\(X_j\). Therefore \(X_j\) again satisfies Hypothesis (H).

Applying
\lemref{lem:one-frequency-absorption} successively, we obtain constants
\[
\gamma_1,\ldots,\gamma_M>0
\]
such that
\[
\|F\|_{L^2(G)}^2
\ge
\gamma_j\|F\|_{L^2(Q)}^2,
\qquad
\forall F\in X_j,
\qquad
j=1,\ldots,M .
\]

After \(M\) steps of the recursion in \eqref{4.48-7-31-w},
 all lower frequencies in \eqref{eq:low-frequency-set} have been added to the high-frequency space. Hence,
\[
	\Lambda_M
	=
	\Lambda_{>N_0}\cup\Lambda_{\le N_0}
	=
	\Sigma_c,
\]
which yields \(X_M=X_{\Sigma_c}\). By the definition of \(X_\Lambda\) and
\eqref{eq:double-spectrum-space}, we obtain
\[
	X_M=\mathcal X .
\]
Therefore
\[
\|F\|_{L^2(G)}^2
\ge
\gamma_M\|F\|_{L^2(Q)}^2,
\qquad
\forall F\in\mathcal X.
\]
This leads to \eqref{eq:full-frequency-X} with
\(C_G=\gamma_M^{-1}\).

It remains to prove \eqref{eq:real-part-QG}. By
\eqref{eq:shifted-free-expansion} and \eqref{eq:double-spectrum-space},
\[
\operatorname{Re}e^{it(\Delta+c)}\varphi\in\mathcal X,
\qquad
\varphi\in L^2(\mathbb T^d;\mathbb C).
\]
Applying \eqref{eq:full-frequency-X} to
\(F=\operatorname{Re}e^{it(\Delta+c)}\varphi\), we obtain
\[
\|\operatorname{Re}e^{it(\Delta+c)}\varphi\|_{L^2(Q)}^2
\le
C_G
\|\operatorname{Re}e^{it(\Delta+c)}\varphi\|_{L^2(G)}^2 .
\]
This proves \eqref{eq:real-part-QG}. Hence
\thmref{thm:double-spectrum-obs} is proved.

\subsection[Proof of the real-part observability theorem]{Proof of \texorpdfstring{\thmref{thm:full}}{the real-part observability theorem}}
  Let \(\phi\in L^2(\mathbb T^d;\mathbb C)\), and set
\[
	v=\operatorname{Re}e^{it(\Delta+c)}\phi .
\]
By \eqref{eq:real-part-QG}, we have
\begin{equation}\label{eq:real-part-spectral-inequality}
	\|v\|_{L^2(Q)}^2
	\le
	C_G\|v\|_{L^2(G)}^2 .
\end{equation}
It remains to estimate \(\|\phi\|_{L^2_x}\) in terms of \(\|v\|_{L^2(Q)}\) and
the stationary purely imaginary component of \(\phi\).

\medskip
\noindent
\emph{Step 1. We separate the stationary and non-stationary modes.}

Set
\begin{equation}\label{eq:mu-and-stationary-set}
	\mu_n=|n|^2-c,
	\qquad
	S=\{n\in\mathbb Z^d:\mu_n=0\}.
\end{equation}
Then \(\mathcal K_c=\ker(\Delta+c)\) is the span of the Fourier modes indexed
by \(S\). Let \(\Pi_S\) be the Fourier projection onto these modes. We write
\begin{equation}\label{eq:phi-Sc-S-decomposition}
	\phi=\phi_{S^c}+\phi_S,
	\qquad
	\phi_S=\Pi_S\phi,
	\qquad
	\phi_{S^c}=(I-\Pi_S)\phi .
\end{equation}
Correspondingly,
\begin{equation}\label{eq:v-Sc-S-decomposition}
	v=v_{S^c}+v_S,
	\qquad
	v_{S^c}=\operatorname{Re}e^{it(\Delta+c)}\phi_{S^c},
	\qquad
	v_S=\operatorname{Re}e^{it(\Delta+c)}\phi_S .
\end{equation}
We shall prove
\begin{equation}\label{eq:stationary-real-part-target}
	\|\operatorname{Re}\phi_S\|_{L^2_x}^2
	=
	(2\pi)^{-1}\|v_S\|_{L^2(Q)}^2,
\end{equation}
and
\begin{equation}\label{eq:nonstationary-control-target}
	\|\phi_{S^c}\|_{L^2_x}^2
	\le
	C_c\|v_{S^c}\|_{L^2(Q)}^2 .
\end{equation}

\medskip
\noindent
\emph{Step 2. We prove \eqref{eq:stationary-real-part-target}.}

By \eqref{eq:mu-and-stationary-set}, the flow does not rotate the modes in
\(S\):
\begin{equation}\label{eq:stationary-flow-identity}
	e^{it(\Delta+c)}\phi_S=\phi_S .
\end{equation}
Combining \eqref{eq:v-Sc-S-decomposition} and
\eqref{eq:stationary-flow-identity}, we obtain
\begin{equation}\label{eq:vS-norm-identity}
	v_S=\operatorname{Re}\phi_S,
	\qquad
	\|v_S\|_{L^2(Q)}^2
	=
	2\pi\|\operatorname{Re}\phi_S\|_{L^2_x}^2 .
\end{equation}
Thus \eqref{eq:stationary-real-part-target} follows from
\eqref{eq:vS-norm-identity}.

\medskip
\noindent
\emph{Step 3. We prove \eqref{eq:nonstationary-control-target}.}

Write
\[
\phi_{S^c}(x)
=
\sum_{k\notin S}a_ke^{ik\cdot x}.
\]
Then
\[
v_{S^c}(t,x)
=
\sum_{k\notin S}
e^{ik\cdot x}
\frac12
\left(
a_ke^{-i\mu_kt}
+
\overline{a_{-k}}e^{i\mu_kt}
\right).
\]
Since \(k\notin S\), we have \(\mu_k\neq0\), so the two time
exponentials in each block are distinct. Applying the uniform block
Riesz estimate \eqref{equ-Riesz-1} and then Parseval's identity in the
spatial variables, we obtain
\[
\begin{aligned}
	\|v_{S^c}\|_{L^2(Q)}^2
	&\geq
	C_c^{-1}
	\sum_{k\notin S}
	\left(
	|a_k|^2+|a_{-k}|^2
	\right)\\
	&\geq
	C_c^{-1}
	\|\phi_{S^c}\|_{L^2_x}^2,
\end{aligned}
\]
where \(C_c>0\) depends only on \(c\). Hence
\[
\|\phi_{S^c}\|_{L^2_x}^2
\leq
C_c\|v_{S^c}\|_{L^2(Q)}^2.
\]

\medskip
\noindent
\emph{Step 4. We prove \eqref{equ-measurable-realpart}.}
 
The function \(v_{S^c}\) has spatial Fourier support in \(S^c\), while
\(v_S\) has spatial Fourier support in \(S\). Hence, by
\eqref{eq:v-Sc-S-decomposition},
\begin{equation}\label{eq:v-orthogonal-decomposition}
	v_{S^c}\perp v_S
	\quad\text{in }L^2(Q),
	\qquad
	\|v\|_{L^2(Q)}^2
	=
	\|v_{S^c}\|_{L^2(Q)}^2
	+
	\|v_S\|_{L^2(Q)}^2 .
\end{equation}
By \eqref{eq:phi-Sc-S-decomposition} and the orthogonality of the Fourier
modes in \(S\) and \(S^c\), we have
\[
\|\phi\|_{L^2_x}^2
=
\|\phi_{S^c}\|_{L^2_x}^2
+
\|\phi_S\|_{L^2_x}^2 .
\]
Moreover,
\[
\|\phi_S\|_{L^2_x}^2
=
\|\operatorname{Re}\phi_S\|_{L^2_x}^2
+
\|\operatorname{Im}\phi_S\|_{L^2_x}^2 .
\]
Using \eqref{eq:nonstationary-control-target},
\eqref{eq:stationary-real-part-target}, and
\eqref{eq:v-orthogonal-decomposition}, and enlarging \(C_c\) if necessary, we
obtain
\begin{equation}\label{eq:pre-real-part-controls-initial-data}
\|\phi\|_{L^2_x}^2
\le
C_c\|v\|_{L^2(Q)}^2
+
\|i\,\operatorname{Im}\phi_S\|_{L^2_x}^2 .
\end{equation}
By the definition of the projection
\(P_{(\mathcal K_c)_{\operatorname{Im}}}\) in
\eqref{eq:P-KcIm-definition}, we have
\[
i\,\operatorname{Im}\phi_S
=
P_{(\mathcal K_c)_{\operatorname{Im}}}\phi .
\]
Combining this identity with \eqref{eq:pre-real-part-controls-initial-data},
we obtain
\begin{equation}\label{eq:real-part-controls-initial-data}
	\|\phi\|_{L^2_x}^2
	\le
	C_c\|v\|_{L^2(Q)}^2
	+
	\|P_{(\mathcal K_c)_{\operatorname{Im}}}\phi\|_{L^2_x}^2 .
\end{equation}
Finally, combining \eqref{eq:real-part-controls-initial-data} with
\eqref{eq:real-part-spectral-inequality}, we obtain
\[
\|\phi\|_{L^2_x}^2
\le
C_cC_G
\|\operatorname{Re}e^{it(\Delta+c)}\phi\|_{L^2(G)}^2
+
\|P_{(\mathcal K_c)_{\operatorname{Im}}}\phi\|_{L^2_x}^2 .
\]
Since \(\phi\in L^2(\mathbb T^d;\mathbb C)\) was arbitrary, this proves
\thmref{thm:full}.

We end this subsection by proving \corref{cor:real-part-iff}.

\begin{proof}[Proof of \corref{cor:real-part-iff}]
We first show that
\[
	c\notin\{|n|^2:n\in\mathbb Z^d\}
\]
implies the observability estimate without correction. In this case
\(\mathcal K_c=\ker(\Delta+c)=\{0\}\), and hence
\[
	P_{(\mathcal K_c)_{\operatorname{Im}}}=0 .
\]
Thus \eqref{equ-measurable-realpart} gives the desired estimate without
correction.

We next show that the observability estimate without correction implies
\[
	c\notin\{|n|^2:n\in\mathbb Z^d\}.
\]
Suppose, by contradiction, that
\(c=|n_0|^2\) for some \(n_0\in\mathbb Z^d\). Set
\[
w(x):=\cos(n_0\cdot x),
\qquad
\phi:=iw.
\]
Then \(0\neq w\in\ker(\Delta+c)\), and hence
\[
e^{it(\Delta+c)}\phi=\phi,
\qquad
\operatorname{Re}e^{it(\Delta+c)}\phi=0
\quad\text{in }L^2(Q;\mathbb R).
\]
The observability estimate without correction would imply
\[
	\|\phi\|_{L^2_x}^2\le 0,
\]
which contradicts \(w\neq0\). This proves the corollary.
\end{proof}

\section{Appendix: a coercivity estimate for nonnegative potentials}
\label{sec:appendix-coercivity}

We include the following elementary quantitative coercivity estimate for
completeness. It is a standard consequence of a low-frequency uncertainty
principle together with the elementary fact that the gradient controls high
frequencies.

We recall the following multidimensional Turán inequality. It will be used to
control low Fourier modes from their mass on a set of positive measure. There
exists a constant \(C_d>0\), depending only on \(d\), such that for every
measurable set \(E\subset\mathbb T^d\) with \(|E|>0\), every \(N\geq1\), and
every trigonometric polynomial
\[
p(x)=\sum_{|n|\leq N} a_n e^{in\cdot x},
\]
one has
\begin{equation}\label{eq:turan-recalled}
	\|p\|_{L^2(E)}^2
	\geq
	\left(
	\frac{|E|}{C_d|\mathbb T^d|}
	\right)^{2C_dN}
	\|p\|_{L^2(\mathbb T^d)}^2 .
\end{equation}
See \cite{Turan-Fontes}.

\begin{lemma}[Coercivity for nonnegative potentials]
	\label{lem-positive-V}
	Assume that
	\[
	V\in C(\mathbb T^d;\mathbb R),
	\qquad
	V\geq0,
	\qquad
	\int_{\mathbb T^d}V(x)\,\d x>0.
	\]
	Then there exists \(\kappa_V>0\) such that
	\[
	\int_{\mathbb T^d}|\nabla u|^2\,\d x
	+
	\int_{\mathbb T^d}V(x)|u|^2\,\d x
	\geq
	\kappa_V
	\|u\|_{L^2(\mathbb T^d)}^2,
	\qquad
	u\in H^1(\mathbb T^d).
	\]
	More precisely, if \(\delta>0\) is such that
	\[
	E_\delta
	:=
	\{x\in\mathbb T^d:V(x)\geq\delta\}
	\]
	has positive measure, then one may take
	\[
	\kappa_V
	=
	\frac12\min\{1,\delta\}
	\left(
	\frac{|E_\delta|}
	{C_d|\mathbb T^d|}
	\right)^{4C_d}.
	\]
\end{lemma}

\begin{proof}
	We first prove a localized observation estimate. Let
	\(E\subset\mathbb T^d\) be measurable with \(|E|>0\). We claim that
	\begin{equation}\label{eq:appendix-observation}
		\int_{\mathbb T^d}|\nabla u|^2\,\d x
		+
		\int_E |u|^2\,\d x
		\geq
		c_E\int_{\mathbb T^d}|u|^2\,\d x,
		\qquad
		\forall u\in H^1(\mathbb T^d),
	\end{equation}
	with
	\[
	c_E=
	\frac12
	\left(
	\frac{|E|}{C_d|\mathbb T^d|}
	\right)^{4C_d}.
	\]
	
	Let \(P_{\leq2}\) be the Fourier projection onto the frequencies
	\(|n|\leq2\). Then
	\[
	P_{\leq2}u=\sum_{|n|\leq2}a_ne^{in\cdot x}
	\]
	if  \( u(x)=\sum_{n\in\mathbb Z^d}a_ne^{in\cdot x}\). 
	Applying \eqref{eq:turan-recalled} to \(P_{\leq2}u\), we obtain
	\[
	\|P_{\leq2}u\|_{L^2(E)}^2
	\geq
	\theta_E\|P_{\leq2}u\|_{L^2(\mathbb T^d)}^2,
	\qquad
	\theta_E:=
	\left(
	\frac{|E|}{C_d|\mathbb T^d|}
	\right)^{4C_d}.
	\]
	After increasing \(C_d\) if necessary, we may assume that
	\(\theta_E\leq1\).
	
	Since \(u=P_{\leq2}u+(I-P_{\leq2})u,\)	the elementary inequality \(|a+b|^2\geq \frac12|a|^2-2|b|^2 \)	gives
	\begin{align*}
		\int_E |u|^2\,\d x
		&\geq
		\frac12\int_E |P_{\leq2}u|^2\,\d x
		-
		2\int_E |(I-P_{\leq2})u|^2\,\d x
		\\
		&\geq
		\frac{\theta_E}{2}
		\|P_{\leq2}u\|_{L^2(\mathbb T^d)}^2
		-
		2\|(I-P_{\leq2})u\|_{L^2(\mathbb T^d)}^2
		\\
		&=
		\frac{\theta_E}{2}
		\|u\|_{L^2(\mathbb T^d)}^2
		-
		\left(2+\frac{\theta_E}{2}\right)
		\|(I-P_{\leq2})u\|_{L^2(\mathbb T^d)}^2 .
	\end{align*}
	The high-frequency term is absorbed by the gradient term. Indeed, since
	\((I-P_{\leq2})u\) contains only frequencies \(|n|>2\),
	\[
	\int_{\mathbb T^d}|\nabla u|^2\,\d x
	\geq
	4\|(I-P_{\leq2})u\|_{L^2(\mathbb T^d)}^2 .
	\]
Since \(\theta_E\leq1\), \(2+\frac{\theta_E}{2}\leq \frac52<4\), and hence
	\[
	\int_{\mathbb T^d}|\nabla u|^2\,\d x
	+
	\int_E |u|^2\,\d x
	\geq
	\frac{\theta_E}{2}
	\|u\|_{L^2(\mathbb T^d)}^2,
	\]
	which proves \eqref{eq:appendix-observation}.
	
	We now apply \eqref{eq:appendix-observation} to \(E_\delta\). Since
	\(V\geq\delta\) on \(E_\delta\),
	\[
	\int_{\mathbb T^d}V(x)|u|^2\,\d x
	\geq
	\delta\int_{E_\delta}|u|^2\,\d x .
	\]
	Therefore,
	\begin{align*}
		\int_{\mathbb T^d}|\nabla u|^2\,\d x
		+
		\int_{\mathbb T^d}V(x)|u|^2\,\d x
		&\geq
		\min\{1,\delta\}
		\left(
		\int_{\mathbb T^d}|\nabla u|^2\,\d x
		+
		\int_{E_\delta}|u|^2\,\d x
		\right)                                                        \\
		&\geq
		\frac12\min\{1,\delta\}
		\left(
		\frac{|E_\delta|}{C_d|\mathbb T^d|}
		\right)^{4C_d}
		\int_{\mathbb T^d}|u|^2\,\d x .
	\end{align*}
	This is the desired estimate.
\end{proof}

\section*{Acknowledgements}
{\small G.W. was partially supported by the National Natural Science Foundation of China (No.~12371450) and the New Cornerstone Science Foundation.  M.W. was partially supported by the National Natural Science Foundation of China (No.~12571260), the Natural Science Foundation of Hunan Province (No.~2026JJ20012), and the Hunan Basic Science Research Center for Mathematical Analysis (2024JC2002). }

\end{document}